\documentclass[oneside]{amsart}

\newtheorem{theorem}[equation]{Theorem}

\newtheorem{conj}[equation]{Conjecture}
\theoremstyle{remark}

\theoremstyle{definition}
\newtheorem{defn}[equation]{Definition}

\numberwithin{equation}{section}

\usepackage[
   marginparwidth=0.83in
  ]{geometry}
\usepackage{graphicx}
\usepackage{amsmath}
\usepackage{amsfonts,amssymb}
\usepackage{amsthm}
\usepackage{verbatim}
\usepackage{array}
\usepackage{graphicx,color}
\usepackage{mathtools}
\usepackage{caption}
\usepackage{subcaption}
\usepackage{fullpage}
\usepackage{hyperref}
\hypersetup{
  colorlinks   = true, 
  urlcolor     = blue, 
  linkcolor    = black, 
  citecolor   = green 
}
\usepackage{physics}
\usepackage{float}
\usepackage{array}
\usepackage{booktabs}
\usepackage{float}
\usepackage{diagbox}
\usepackage{enumitem}

\newcommand{\Span}{\operatorname{span}}

\newcommand{\Id}{\operatorname{Id}}

\definecolor{green(html/cssgreen)}{rgb}{0.0, 0.5, 0.0}

\newcommand{\Real}{\mathbb{R}}
\newcommand{\Com}{\mathbb{C}}
\newcommand{\ademi}{\mbox{$\frac{a}{2}$}}
\newcommand{\bdemi}{\mbox{$\frac{b}{2}$}}
\newcommand{\der}{\mathrm{d}}

\begin{document}
\title{Numerical optimisation of Dirac eigenvalues}

\author[Antunes, Bento and Krej\v{c}i\v{r}\'{\i}k]{
Pedro R.~S. Antunes, Francisco Bento and David Krej\v{c}i\v{r}\'{\i}k}

\address{Department of Mathematics and Group of Mathematical Physics, Instituto Superior T\'{e}cnico, Av.~Rovisco~Pais~1, 1049-001 Lisboa, Portugal} \email{prsantunes@tecnico.ulisboa.pt}

\address{Department of Mathematics, Instituto Superior T\'{e}cnico, Av.~Rovisco~Pais~1, 1049-001 Lisboa, Portugal} \email{francisco.a.bento@tecnico.ulisboa.pt}

\address{Department of Mathematics, Faculty of Nuclear Sciences and 
Physical Engineering, Czech Technical University in Prague, 
Trojanova 13, 12000 Prague 2, Czech Republic}
\email{David.Krejcirik@fjfi.cvut.cz}

\date{March 27, 2024}

\begin{abstract}
Motivated by relativistic materials,
we develop a numerical scheme to support existing 
or state new conjectures in the spectral optimisation
of eigenvalues of the Dirac operator,
subject to infinite-mass boundary conditions.
We study the optimality 
of the regular polygon (respectively, disk)
among all polygons of a given number of sides 
(respectively, arbitrary sets),
subject to area or perimeter constraints.
We consider the three lowest positive eigenvalues and their ratios.
Roughly, we find results analogous 
to known or expected for the Dirichlet Laplacian,
except for the third eigenvalue which does not need
to be minimised by the regular polygon 
(respectively, the disk) for all masses.
In addition to the numerical results, 
a new, mass-dependent upper bound to the lowest  
eigenvalue in rectangles is proved
and its extension to arbitrary quadrilaterals is conjectured. 
\end{abstract}

\maketitle

\vspace{-6ex}

\section{Introduction}
%
\subsection{Prologue}
The ancient Greeks had no doubt that the disk has the largest 
area among all domains of fixed perimeter.
Simply because they could have checked it by measuring 
on a huge number of samples.
The conviction is reflected in the celebrated account 
of the legendary Dido founding Carthage almost a millennium~BC.
Moreover, the Greeks pretty much solved the isoperimetric problem, 
by their standards,
when Zenodorus proved that the disk has greater area than any
polygon with the same perimeter.
However, a rigorous proof for all domains
was done only by Steiner, by his standards,
and completed by Weierstrass in 1879. 
The almost three millennia were necessary to develop 
new tools in mathematics to prove 
the isoperimetric inequality~\cite{Blasjo_2005}.

The objective of the present paper is to give an account 
on a problem in contemporary mathematics which puts us
in a situation much similar to that of the ancient Greeks.
By developing a robust numerical machinery, 
there is no doubt that a spectral isoperimetric problem holds true.
However, no rigorous proof is available 
because of a fundamental lack of appropriate tools in mathematics. 

\subsection{The problem}
Motivated by recent advances in nanotechnologies,
there has been a surge of interest in
relativistic materials such as graphene.
Mathematically, the challenge consists in replacing 
the traditional scalar Laplacian 
$-\Delta = -\partial_1^2-\partial_2^2$
on a planar open set~$\Omega$
by the Dirac operator 
\begin{equation}\label{operator}
  H =
  \begin{bmatrix}
    m & -i(\partial_1 - i \partial_2)\\
    -i(\partial_1 + i \partial_2) & -m
  \end{bmatrix}
  \qquad\mbox{in}\qquad
  L^2(\Omega;\Com^2)
  \,.
\end{equation}
Here $m \geq 0$ denotes the (effective) mass
of the \mbox{(quasi-)}particle modelled.
Since~$H$ is matrix-valued
and unbounded both from above and below, 
the powerful mathematical tools such as 
variational methods and the maximum principle are not directly available. 
Moreover, non-trivial geometries  
additionally lead to exotic boundary conditions, 
among these the distinguished infinite-mass constraint  
\cite{Arrizibalaga-LeTreust-Raymond_2017,
Arrizibalaga-LeTreust-Mas-Raymond_2019,
Barbaroux-Cornean-LeTreust-Stockmeyer_2019}
\begin{equation}\label{bc}
  u_2 = i (n_1 + i n_2)u_1
  \qquad \mbox{on} \qquad
  \partial\Omega
  \,,
\end{equation}
which is commonly accepted as the relativistic 
counterpart of Dirichlet boundary conditions 
for the Laplacian.
Here 
$ 
  u =
  \begin{bsmallmatrix} u_1\\ u_2 \end{bsmallmatrix}
$
is a function from the domain of~$H$ 
and \(n = (n_1, n_2)\) 
stands for the outward unit normal of~\(\Omega\).
Henceforth, we assume that~$\Omega$ is a bounded Lipschitz open set, 
to have~$n$ well defined
and~$H$ self-adjoint~\cite{Behrndt}. 

The lack of mathematical apparatus leads to
unresolved problems in spectral geometry of optimal shapes.
In particular, the following notoriously well-known facts
for the Dirichlet Laplacian remain a mystery for 
the Dirac operator, subject to infinite-mass boundary conditions:
\begin{description}[align=right,labelwidth=3.5em]
\item[\textbf{(Q1)}]\label{Q1}
Among all sets of fixed area, 
does the disk minimises the first eigenvalue?
\item[\textbf{(Q1')}]\label{Q1_p}
Among all sets of fixed perimeter, 
does the disk minimises the first eigenvalue?
\item[\textbf{(Q2)}]\label{Q2}
Among all quadrilaterals of fixed area, 
does the square minimises the first eigenvalue?
\item[\textbf{(Q2')}]\label{Q2_p}
Among all quadrilaterals of fixed perimeter, 
does the square minimises the first eigenvalue?
\end{description}
Here the eigenvalues are counted starting from
the first \emph{positive} eigenvalue of~$H$.
This is justified by the fact that the spectrum of~$H$ 
is symmetric with respect to zero
and zero is never an eigenvalue.

In the non-relativistic setting, 
the validity of~\hyperref[Q1]{\textbf{(Q1)}} is known as the Faber--Krahn inequality
\cite{Faber_1923,Krahn_1924},
while its discrete analogue~\hyperref[Q2]{\textbf{(Q2)}} is due to P\'{o}lya and Szeg\H{o} 
\cite{pol-szego}.
The isoperimetric constraints are considered to be simpler;
indeed, the validity of \hyperref[Q1_p]{\textbf{(Q1')}} and \hyperref[Q2_p]{\textbf{(Q2')}} follow from \hyperref[Q1]{\textbf{(Q1)}} and \hyperref[Q2]{\textbf{(Q2)}}
by scaling, respectively
(and, historically, Courant~\cite{Courant_1918} established~\hyperref[Q1_p]{\textbf{(Q1')}} 
for the Dirichlet Laplacian before~\hyperref[Q1]{\textbf{(Q1)}} of Faber and Krahn).

In the present relativistic setting,
the validity of~\hyperref[Q1]{\textbf{(Q1)}} for massless particles (i.e.\ $m=0$)
is explicitly stated as a conjecture in~\cite{AKO}
(see also \cite{krejcirik_larson_lotoreichik_2019}).
The conjectures about the validity of~\hyperref[Q1]{\textbf{(Q1)}} and \hyperref[Q1_p]{\textbf{(Q1')}}
for all masses $m \geq 0$, are explicitly stated in~\cite{briet-krej}. 
For recent attempts to prove the validity of~\hyperref[Q1]{\textbf{(Q1)}}, 
see \cite{Benguria-Fournais-Stockmeyer-Bosch_2017,
Lotoreichik-Ourmieres_2019,AKO,
Arrizabalaga-Mas-Sanz-Perela-Vega_2023,
Behrndt-Frymark-Holzmann-Stelzer-Landauer}. 
Since the isochoric result~\hyperref[Q1]{\textbf{(Q1)}} is not available,
\hyperref[Q1_p]{\textbf{(Q1')}} is independently interesting.

In the last named reference~\cite{briet-krej}, 
the conjectures about the validity of \hyperref[Q2]{\textbf{(Q2)}} and~\hyperref[Q2_p]{\textbf{(Q2')}}
are explicitly stated for rectangles
and partial results are established.
The restriction to rectangles does not simplify the problem, 
for the eigenvalue equation cannot be solved by separation of variables.
It is frustrating that such an illusively simple question 
does not appear to be answerable by current mathematical tools.
Note that the question for polygons of more sides than four
remains open in the non-relativistic setting, too
\cite{pol-szego,Bogosel-Bucur_2023,Indrei}.

\subsection{The results}
Because of the fundamental lack of rigorous tools to tackle
the questions \hyperref[Q1]{\textbf{(Q1)}}, \hyperref[Q1_p]{\textbf{(Q1')}}, \hyperref[Q2]{\textbf{(Q2)}} and \hyperref[Q2_p]{\textbf{(Q2')}} for the Dirac operator,
in this paper we overtake the pragmatical attitude 
of the ancient Greeks and study the problems numerically.
In fact, the affirmative answer to question~\hyperref[Q1]{\textbf{(Q1)}} 
was numerically supported already in~\cite{AKO}. 
At the same time, as mentioned above, the validity of \hyperref[Q1_p]{\textbf{(Q1')}} and \hyperref[Q2_p]{\textbf{(Q2')}} 
follow as a consequence of \hyperref[Q1]{\textbf{(Q1)}} and \hyperref[Q2]{\textbf{(Q2)}}, respectively. 
Therefore, we focus on question~\hyperref[Q2]{\textbf{(Q2)}} for quadrilaterals.

The numerical machinery that we employ is  
the Method of Fundamental Solutions.
This is a \emph{meshless} technique used to approximate solutions 
of elliptic partial differential equations
using the fundamental solutions.
Previously, this method has demonstrated to be effective
in studying the spectrum of elliptic operators and solving inverse problems, 
see, e.g., \cite{alves2005method,barnett2008stability,alves2013method}.

Our main numerical results confirm the validity of~\hyperref[Q2]{\textbf{(Q2)}} for quadrilaterals. 
In particular, we provide affirmative answers to 
the conjectures raised in~\cite{briet-krej} for rectangles.
At the same time, we establish a new analytic upper bound 
to the first eigenvalue and numerically test it.
This is an improvement to the mass-independent 
upper bound of~\cite{briet-krej}.
It turns out that, numerically, the improved bound 
extends to general quadrilaterals.
Moreover, we provide a numerical support for the validity 
of the relativistic P\'{o}lya--Szeg\H{o} inequality in general,
at least for octagons and all the polygons of less sides.

We also investigate higher eigenvalues. 
In agreement with the known fact for the Dirichlet Laplacian, 
we numerically show that the second eigenvalue 
of the Dirac operator is minimised by two disjoint identical disks,
among all sets of fixed area. 
What is more, and still agreement with the non-relativistic setting,
it is numerically argued that the ratio of the second 
to the first eigenvalue is minimised by the disk,
obtaining in this way a relativistic analogue of 
the celebrated Ashbaugh--Benguria result~\cite{ash-ben}. 

In agreement with the known fact for the Dirichlet Laplacian, 
we numerically show that the ratio of the third to the first eigenvalue
is not maximised by the disk; 
instead, a familiar peanut-like shape is found
\cite{Ant,Levitin,osting}.
As for the third eigenvalue by itself, 
our numerical experiments show that 
it is not minimised by the disk for every~$m$,
among all sets of fixed area. 
This is a novelty with respect to the Dirichlet Laplacian,
for which it is expected that the third eigenvalue 
is minimised by the disk \cite[Sec.~5.3]{henrot2006extremum}. 
We also numerically show that the third eigenvalue is not minimised 
by the square for every~$m$, 
among all rectangles of fixed area.

\subsection{The structure}
The paper is organised as follows.
Our numerical scheme is developed in Section~\ref{Sec.method}.
The results for rectangles (as well as more general quadrilaterals),
triangles, general polygons and domains are presented in
Sections~\ref{Sec.rect}, \ref{Sec.tri}, \ref{Sec.poly}
and~\ref{Sec.general}, respectively.

\subsection{Epilogue}
The conclusions of our numerical optimisations of Dirac eigenvalues 
are summarised in new relativistic conjectures stated within the text.
We hope that their proofs will be achieved 
within a time period significantly less than three millennia. 
However, it is important to keep in mind that
it is also possible that some of the conjectures do not hold,
for they are merely based on numerics.

\section{The numerical method}\label{Sec.method}
%
\subsection{Preliminaries}
First of all, let us comment on the definition of~$H$
in more detail. 
Unless otherwise stated, the mass~$m$ is assumed to be
an arbitrary non-negative number.
Given a planar open connected set $\Omega$, 
the Hilbert space is $L^2(\Omega;\Com^2)$,
the space of two-tuples (spinors)
$ 
  u =
  \begin{bsmallmatrix} u_1\\ u_2 \end{bsmallmatrix}
$
composed of complex-valued square-integrable functions $u_1,u_2$.
The action of~$H$ in~\eqref{operator} should be 
understood in the sense of distributions.
The boundary conditions~\eqref{bc}
are incorporated through the operator domain 
on which~$H$ is self-adjoint.
If~$\Omega$ is merely a Lipschitz open set with compact boundary,
then~$H$ is self-adjoint on the domain composed of
functions    
$ 
  u =
  \begin{bsmallmatrix} u_1\\ u_2 \end{bsmallmatrix}
  \in H^{1/2}(\Omega;\Com^2)
$
such that $Hu \in L^2(\Omega;\Com^2)$ 
and~\eqref{bc} holds 
in the sense of (generalised) traces~\cite{Behrndt}. 
For more regular sets~$\Omega$, 
like smooth sets or convex polygons,
the domain turns out to be composed of   
$ 
  u =
  \begin{bsmallmatrix} u_1\\ u_2 \end{bsmallmatrix}
  \in H^{1}(\Omega;\Com^2)
$
such that~\eqref{bc} holds in the sense
of (classical) traces,
see~\cite{Benguria-Fournais-Stockmeyer-Bosch_2017b}
and~\cite{LeTreust-Ourmieres-Bonafos_2018}, respectively. 

From now on, we assume that~$\Omega$ is 
a bounded Lipschitz open connected subset of~$\Real^2$.
Consequently, $H$~is an operator with compact resolvent,
so its spectrum is purely discrete.  
Zero is never an eigenvalue, for $H u=0$
has only a trivial solution.
So the eigenvalues of $H$
(counted with multiplicities) compose of the union of 
a non-increasing sequence $\{\lambda_{-j}\}_{j=1}^\infty$
of negative numbers 
and a non-decreasing sequence $\{\lambda_{j}\}_{j=1}^\infty$
of positive numbers,
which accumulate at~$-\infty$ and~$+\infty$, respectively.  
The charge conjugation symmetry 
ensures that the spectrum 
is symmetric with respect to zero, 
i.e.\ $\lambda_{-j} = \lambda_j$ for every $j \geq 1$.
Indeed,  
$ 
  u =
  \begin{bsmallmatrix} u_1\\ u_2 \end{bsmallmatrix}
$
is an eigenfunction corresponding to an eigenvalue~$\lambda$
if, and only if,  
$ 
  u =
  \begin{bsmallmatrix} \bar{u}_2\\ \bar{u}_1 \end{bsmallmatrix}
$
is an eigenfunction corresponding to an eigenvalue~$-\lambda$.
In summary,
$$
  \sigma(H) =
  \{ -\infty \leftarrow \dots \leq -\lambda_2 \leq -\lambda_1 
  < \lambda_1 \leq \lambda_2 \leq \dots \to +\infty\}
  \,,
$$
where~$\lambda_1$ is positive.

\subsection{The method}
Now, let us consider the eigenvalue equation
$
  H u = \lambda u 
$.
This is a system of two partial differential equations.
Expressing the second component of~$u$ 
from the second equation via 
(note that $\lambda = \pm m$ is never an eigenvalue)
\[
  u_2 = -\frac{i (\partial_1 + i \partial_2)u_1}{\lambda + m}
\]
and plugging this result to the first equation,
we arrive at the Helmholtz equation with Cauchy--Riemann 
oblique boundary conditions 
\begin{equation}\label{dirac_helmholtz}
\left\{
  \begin{aligned}
    &- \Delta u_1 = (\lambda^2 - m^2)u_1 
    &&\text{in} \quad \Omega \,,
    \\
    &i(\partial_1 + i\partial_2)u_1 + i(\lambda+m)(n_1+in_2)u_1 = 0 
    &&\text{on} \quad \partial\Omega
    \,.
  \end{aligned}
\right.  
\end{equation}
A similar problem can be obtained for~$u_2$.

The problem~\eqref{dirac_helmholtz} is well adapted 
to the Method of Fundamental Solutions (MFS), 
see, e.g., \cite{alves2005method} and \cite{bogolmony}. 
In the present case, 
the fundamental solution of the Helmholtz equation in \(\mathbb{R}^2\) is given by
\[
  \Phi_\lambda(x) = \frac{i}{4}H_0^{(1)} (\sqrt{\lambda}\norm*{x}).
\]
Here \(H_0^{(1)}\) is the Hankel function of the first kind and order zero 
and one has
\[
  H_0^{(1)} = J_0(x) + iY_0(x),
\]
where \(J_0\) and \(Y_0\) are the Bessel functions of first and second kind with order zero, respectively.

In order to approximate the pair of eigenvalues/eigenvectors in \eqref{dirac_helmholtz}, first let \(\hat{\Gamma}\) be a so-called \emph{admissible source set}, as defined in \cite{calves_source}, for instance the boundary of an open set \(\hat{\Omega}\) such that \(\Omega \subset \hat{\Omega}\) and \(\hat{\Gamma}\) surrounds \(\partial\Omega\). Then, we consider an expansion \(u_N\) of the form
\[
  u_N(x) = \sum_{j=1}^{N}\alpha_j \Phi_\lambda(x - y_j),
\]
where \(y_j \in \hat{\Gamma},\ j=1,\dots, N\) are known as \emph{source points}. We place them using the standard technique presented in \cite{alves2005method} based on the outward unit normal to \(\partial\Omega\) with displacement \(\eta\). 
That is, each source point \(y_j\) is defined by \(y_j = x_i +\eta n_{x_i}\), 
where \(x_i\in\partial\Omega\) and \(n_{x_i}\) is the outward unit normal to the boundary at the point \(x_i\).

Let \(\mathcal{B}\) be a linear boundary operator such that
\[
  \mathcal{B}u = 0 \text{ on } \partial\Omega,
\]
and let \(x_1,\dots,x_M\) be a discretisation of the boundary \(\partial\Omega\). Then, by linearity,
\begin{equation*}
  \mathcal{B}u(x_i) = \sum_{j=1}^{N}\alpha_j \mathcal{B}\Phi_\lambda(x_i-y_j) = 0, \quad i=1,\dots,M,
\end{equation*}
from which we can extract the coefficients \(\alpha_j\) either by solving an interpolation problem (a linear system if \(N=M\)) or by least-squares (if \(M > N\)) given by
\begin{equation}\label{MFS_m_system}
  {\underbrace{\begin{bmatrix}
      \mathcal{B}\Phi_\lambda(x_1-y_1) & \cdots & \mathcal{B}\Phi_\lambda(x_1-y_N) \\
      \vdots & \ddots & \vdots\\
      \mathcal{B}\Phi_\lambda(x_M-y_1) & \cdots & \mathcal{B}\Phi_\lambda(x_M-y_N)
  \end{bmatrix}}_{A_1(\lambda)}}
  {\underbrace{\begin{bmatrix}
      \alpha_1\\
      \vdots\\
      \alpha_N
  \end{bmatrix}}_\alpha}
  =
  \begin{bmatrix}
      0\\
      \vdots\\
      0
  \end{bmatrix}.
\end{equation}

For example, as proved in \cite{calves_source}, 
if \(\mathcal{B} = \Id\), i.e., 
the classical Dirichlet problem
\begin{equation*}
\left\{
  \begin{aligned}
    - \Delta u &= \lambda u &&\text{in} \quad \Omega,\\
    u &= 0 &&\text{on} \quad \Omega,
  \end{aligned}
\right.  
\end{equation*}
is considered,
then we have the density result
\[
  H^s(\partial\Omega) = \overline{\Span\left\{\Phi_\lambda(x - y_j)_{|x \in \partial\Omega}: y \in \hat{\Gamma}\right\}}
\]
for every \(s \leq \frac{1}{2}\). An adaptation of this density result for the Cauchy--Riemann oblique boundary conditions can be deduced using layer potential's jump results in conjunction with the Hahn--Banach theorem.

The eigenvalues~$\lambda$ 
can be calculated by the Subspace Angle Technique (SAT) introduced in \cite{mps}. Let \(\mathcal{I}\) be a set of \(L\) interior points to \(\Omega\), randomly sampled using Halton sequences as in \cite{halton1964radical}. 
Then we compute the matrix 
\begin{equation*}
  A_2(\lambda) = \begin{bmatrix}
      \Phi_\lambda(z_1-y_1) & \cdots & \Phi_\lambda(z_1-y_N) \\
      \vdots & \ddots & \vdots\\
      \Phi_\lambda(z_L-y_1) & \cdots & \Phi_\lambda(z_L-y_N)
  \end{bmatrix}
\end{equation*}
and define the block matrix 
\[
  \mathbf{A}(\lambda) = \begin{bmatrix}
    A_1(\lambda)\\
    A_2(\lambda)
  \end{bmatrix}
\]
with dimensions \((M+L) \times N\). Finally, the QR factorisation of \(\mathbf{A}\) is computed, obtaining the matrix \(\mathbf{Q}(\lambda)\). To approximate the eigenvalues of \(H\) in \eqref{dirac_helmholtz}, we study the smallest singular value \(s(\lambda)\) of \(\mathbf{Q}\). In particular, we look for the local minimisers of \(s(\lambda)\) (which should be approximately zero if the eigenfunctions are smooth), that are the desired eigenvalues \(\lambda\) of \(H\), ordered by magnitude. To bracket each local minima, a variation of the Golden Ratio Search described in \cite{alves2005method} is implemented.

\subsection{The disk}
To finish this section, we validate the numerical method in the unit disk \(\mathbb{D}\) for \(m=0\).
In this case, the spectral problem can be solved explicitly 
by separation of variables 
in polar coordinates~\cite{Lotoreichik-Ourmieres_2019}. 
It turns out that 
the first eigenvalue \(\lambda_1(\mathbb{D})\) 
is the smallest positive solution of
\[
  J_0(\lambda_1(\mathbb{D})) = J_1(\lambda_1(\mathbb{D})),
\] 
which is approximately \(\lambda_1(\mathbb{D}) \approx 1.434695650819\). 

\begin{table}[h]
  \centering
  \begin{tabular}{cccc}
      \toprule
      \textbf{First eigenvalue} & \textbf{N=600} & \textbf{N=800} & \textbf{N=1000} \\
      \midrule
      \(1.434695650819\) & $1.434695644389$ & $1.434695652496$ & $1.434695653680$ \\
      \midrule
      \textbf{Absolute error: } \(\abs{\lambda_1(\mathbb{D}) - \tilde{\lambda}_1(\mathbb{D})}\) & $6.429\times 10^{-9}$ & $1.677\times 10^{-9}$ & $2.861\times 10^{-9}$ \\
      \bottomrule
  \end{tabular}
  \caption{The first eigenvalue of $\mathbb{D}$ for $m=0$
  and approximate values and absolute errors obtained for different values of \(N\).}
  \label{tab_eigenvalues_disk_val}
\end{table}

In Table \ref{tab_eigenvalues_disk_val}, we provide a summary of our numerical results for the approximation of the first eigenvalue \(\tilde{\lambda}_1(\mathbb{D})\) using the MFS, with varying numbers of source points \(N\). In this context, the number of collocation points is set to \(M = 2N\), the quantity of inner collocation points utilised for the SAT is \(L=228\), and the displacement parameter is \(\eta=0.05\). Although precision for the disk can be increased by adjusting \(\eta\), we selected this value as it was consistently employed in the majority of simulations, particularly in cases involving non-smooth domains or domains with cusps. The presented results remain highly consistent with this configuration of boundary, inner, and source points. 
Figure~\ref{eigen_disk_m_0_u} shows the absolute value of 
the eigenfunction components~\(u_1\) and~\(u_2\)
associated with the first eigenvalue in the unit disk with \(m=0\).

\begin{figure}[h]
\begin{center}
\begin{tabular}{cc}
    \includegraphics[width=0.48\linewidth]{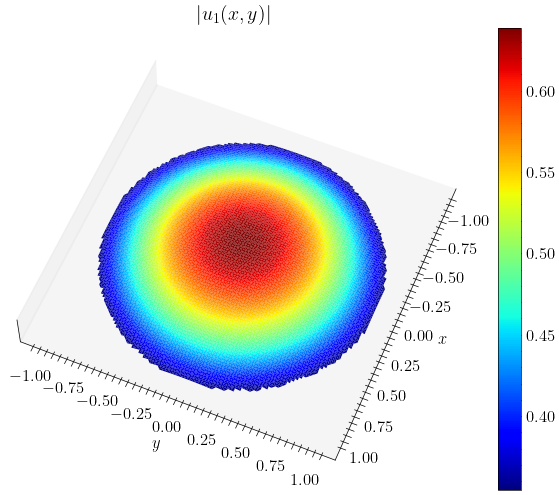}
    &
    \includegraphics[width=0.48\linewidth]{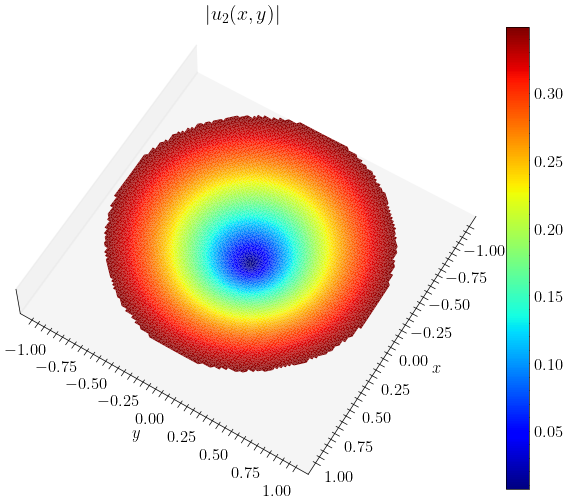}
\end{tabular}
\caption{Absolute values of the first eigenfunction components 
of $\mathbb{D}$ for $m=0$.}\label{eigen_disk_m_0_u}  
\end{center}  
\end{figure}

For future reference, 
Table~\ref{tab_eigenvalues_unit_area_disk} shows the values of 
the first three eigenvalues and its ratios for different values of \(m\), 
when considering the disk with unit area. 

\begin{table}[h]
  \centering
  \renewcommand{\arraystretch}{1.25} 
  \begin{tabular}{c|c|c|c|c|c}
      \diagbox{\textbf{Masses}}{\textbf{Eigenvalues}} & \(\lambda_1\) & \(\lambda_2\) & \(\lambda_3\) & \(\frac{\lambda_2}{\lambda_1}\) & \(\frac{\lambda_3}{\lambda_1}\) \\
      \hline
      \(m = 1\) & 3.129162417 & 5.140771507 & 5.739242446 & 1.642858638 & 1.834114591 \\
      \hline
      \(m = 5\) & 6.195931901 & 7.696248341 & 7.851334614 & 1.242145405 & 1.267175744 \\
      \bottomrule
  \end{tabular}
  \caption{Eigenvalues of $\mathbb{D}$ 
  and their ratios
  for different values of~$m$.}
  \label{tab_eigenvalues_unit_area_disk}
\end{table}
%

\section{Rectangles}\label{Sec.rect}
%
We start the presentation of our numerical results 
by considering rectangles.
Contrary to the case of the Dirichlet 
(or, more generally, Robin~\cite{Laugesen_2019}) Laplacian,
the Dirac problem on rectangles is not explicitly solvable 
by separation of variables~\cite{briet-krej}.
At the same time, no symmetrisation methods are available
in the case of the Dirac operator.
Therefore, the following conjecture due to~\cite{briet-krej}
(which is a special case of questions~\hyperref[Q2]{\textbf{(Q2)}} and~\hyperref[Q2_p]{\textbf{(Q2')}} from the introduction),
despite its illusive simplicity,
actually represents a hard open problem in spectral geometry.

\begin{conj}[\cite{briet-krej}]\label{rectangular_conjecture}
  Let \(\lambda_1 (a, b) \) denote the first eigenvalue of~$H$ 
  on a rectangle of sides \(a\) and \(b\). Then
  \begin{itemize}
    \item Area constraint (\textit{unitary area}): \(\lambda_1(a, \frac{1}{a}) \geq \lambda(1, 1), \, \forall a > 0\);
    \item Perimeter constraint (\textit{perimeter equal to 4}): \(\lambda_1(a, 2 - a) \geq \lambda(1, 1), \, \forall a \in (0, 2)\).
  \end{itemize}
\end{conj}

In \cite{briet-krej}, Conjecture \ref{rectangular_conjecture} 
was proved under additional assumptions
(namely for large eccentricity of the rectangles or large masses)
as a consequence of the lower and upper bounds given 
by the following theorem.

\begin{theorem}[\cite{briet-krej}]\label{laplacian_bound}
  For every \(m \geq 0\), one has 
  \begin{equation}\label{laplacian_bound_eqn}
    \left(\frac{\pi}{a}\right)^2 \max\left\{\frac{1}{1+(ma)^{-1}}, \frac{1}{2}\right\}^2 + \left(\frac{\pi}{b}\right)^2 \max\left\{\frac{1}{1+(mb)^{-1}}, \frac{1}{2}\right\}^2 \leq \lambda_1(a, b)^2 - m^2 \leq \left(\frac{\pi}{a}\right)^2 + \left(\frac{\pi}{b}\right)^2.
  \end{equation}
\end{theorem}

In this paper, we present a strong numerical evidence 
that no additional hypotheses are needed for the validity of 
Conjecture~\ref{rectangular_conjecture}.
Moreover, higher eigenvalues are equally investigated.
In particular,
Figures~\ref{eigen_rectangle_area_m_1}--\ref{eigen_rectangle_area_m_5}
and 
\ref{eigen_rectangle_perimeter_m_1}--\ref{eigen_rectangle_perimeter_m_5} 
represent the dependence of the first five eigenvalues on~$a$,
for various values of the mass (namely \(m = 1\) and \(m = 5\)),
under the area or perimeter constraint, respectively. 
Several interesting observations can be made.

\begin{figure}[h]
  \centering
  \begin{minipage}{.5\textwidth}
    \centering
    \includegraphics[width=0.9\linewidth]{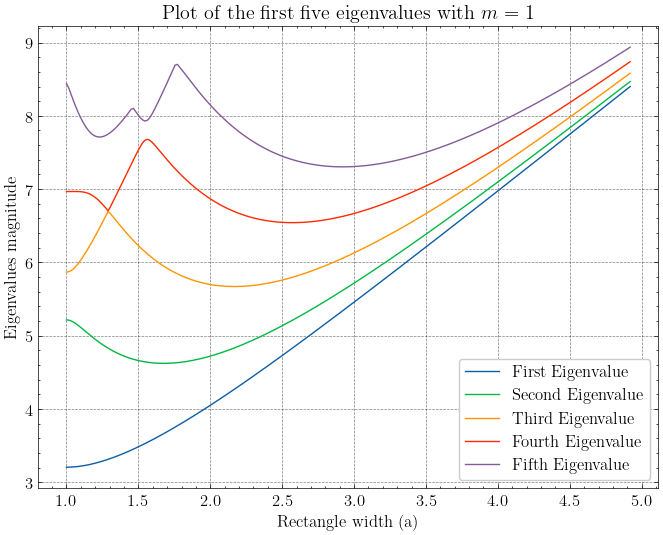}
    \captionsetup{width=0.9\linewidth}
    \captionof{figure}{Behaviour of the first five eigenvalues for
    rectangles with unit area, width \(a\) and \(m=1\).}
    \label{eigen_rectangle_area_m_1}
  \end{minipage}%
  \begin{minipage}{.5\textwidth}
    \centering
    \includegraphics[width=0.9\linewidth]{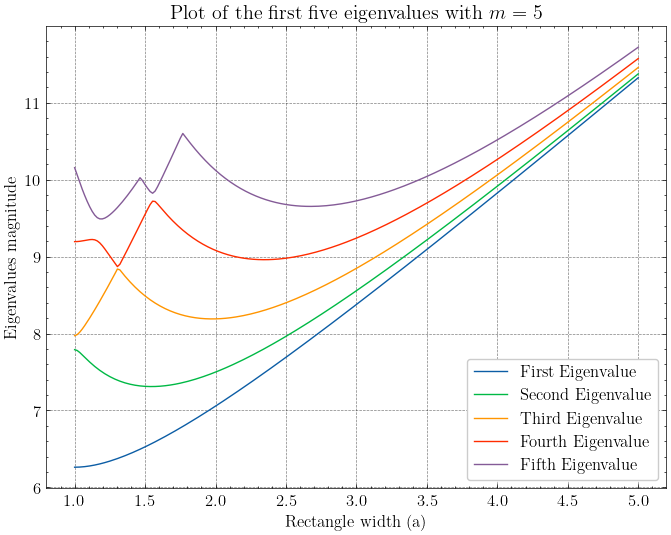}
    \captionsetup{width=0.9\linewidth}
    \captionof{figure}{Behaviour of the first five eigenvalues for
    rectangles with unit area, width \(a\) and \(m=5\).}
    \label{eigen_rectangle_area_m_5}
  \end{minipage}
\end{figure}

\begin{figure}[h]
  \centering
  \begin{minipage}{.5\textwidth}
    \centering
    \includegraphics[width=0.9\linewidth]{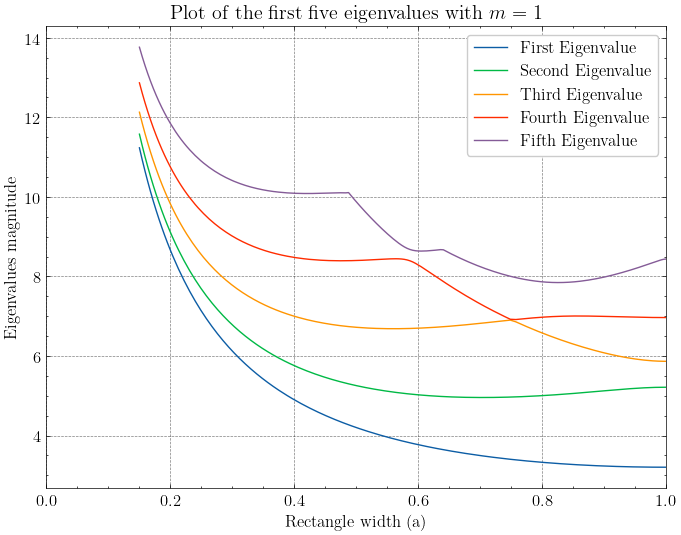}
    \captionsetup{width=0.9\linewidth}
    \captionof{figure}{Behaviour of the first five eigenvalues for
    rectangles with perimeter equal to 4, width \(a\) and \(m=1\).}
    \label{eigen_rectangle_perimeter_m_1}
  \end{minipage}%
  \begin{minipage}{.5\textwidth}
    \centering
    \includegraphics[width=0.9\linewidth]{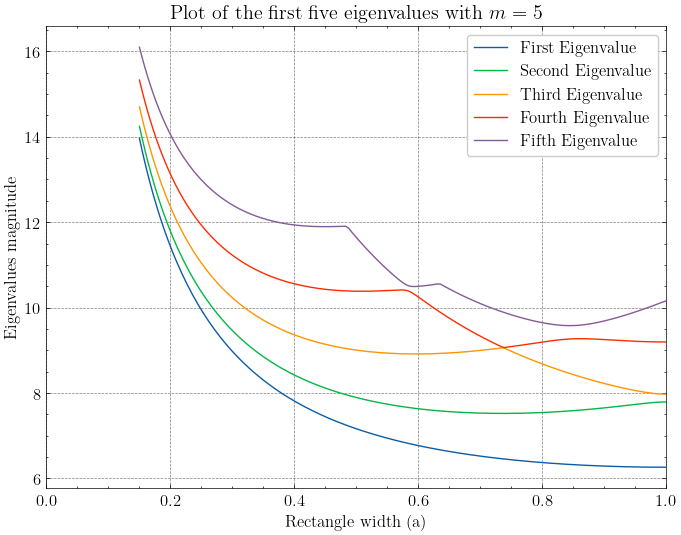}
    \captionsetup{width=0.9\linewidth}
    \captionof{figure}{Behaviour of the first five eigenvalues for
    rectangles with perimeter equal to 4, width \(a\) and \(m=5\).}
    \label{eigen_rectangle_perimeter_m_5}
  \end{minipage}
\end{figure}

\subsection{The area constraint}
First, let us consider the area constraint 
($b=\frac{1}{a}$ and $a>1$ without loss of generality). 
The square (corresponding to $a=1$) is clearly the global minimum 
for the first eigenvalue in 
Figures~\ref{eigen_rectangle_area_m_1} and~\ref{eigen_rectangle_area_m_5}.

While the behaviour of the spectrum remains qualitatively consistent for different values of \(m\), changes in the eigenvalues' magnitude 
are evident. Significantly, as \(m\) increases, the eigenvalues approach \(m\) with a diminishing gap. While this trend is already discernible 
in Figures~\ref{eigen_rectangle_area_m_1} and~\ref{eigen_rectangle_area_m_5}, it becomes more apparent as \(m\) takes on even larger values. 
This behaviour is consistent with the asymptotics 
\begin{equation}\label{limit}
  \lambda_1(a, b)^2 - m^2 \xrightarrow[m \to \infty]{} 
  \left(\frac{\pi}{a}\right)^2 + \left(\frac{\pi}{b}\right)^2
  \,,
\end{equation}
which follows from Theorem~\ref{laplacian_bound}.
Note that the right-hand side is the first eigenvalue
of the Dirichlet Laplacian in the rectangle of sides~$a$ and~$b$.
This ``non-relativistic limit'' is well known to hold for smooth domains
\cite{Arrizibalaga-LeTreust-Raymond_2017,
Behrndt-Frymark-Holzmann-Stelzer-Landauer}.

In agreement with the Dirichlet Laplacian, 
the second eigenvalue is not minimised by the square.
On the other hand, as a new result, 
different from what happens for the Dirichlet Laplacian,
we observe that
the third eigenvalue is not globally minimised by the square;
if \(m = 1\), the minimiser is the rectangle of width \(a \approx 2.2\). 

Spikes are evident in the plots. 
For instance, the third and fourth eigenvalues appear to coincide for \(a\) between the values 1 and 1.5. 
These spikes correspond to domains where the eigenvalue has multiplicity 2. This behaviour becomes more frequent as the order of the eigenvalues increases. Finally, by increasing the value of the parameter \(a\), an approximate linear growth in the magnitude of the eigenvalues is observed, with no change in their multiplicity.
 
\subsection{The perimeter constraint}
In Figures \ref{eigen_rectangle_perimeter_m_1} and \ref{eigen_rectangle_perimeter_m_5} the results with fixed perimeter 
are shown ($b=2-a$ and $a \in (0,1)$ without loss of generality). 
Again, the square (corresponding to $a=1$) is clearly 
the global minimum for the first eigenvalue.

Contrary to the area constraint, 
there is no particularity in the behaviour of the third eigenvalue. 
In particular,  \(\lambda_3(a, b)\) is always minimised 
by the square among all rectangles with a fixed perimeter.

\subsection{Eigenfunctions}
The present numerical method enables one to compute 
the eigenfunctions of the Dirac operator, too.
As an illustration, 
Figure~\ref{eigenfunction_abs_u_1_square} shows
the absolute value of the eigenfunction components~$u_1$ and~$u_2$
associated with the first eigenvalue in the unit square 
with $m=1$. 

\begin{figure}[h]
\begin{center}
\begin{tabular}{cc}
    \includegraphics[width=0.48\linewidth]{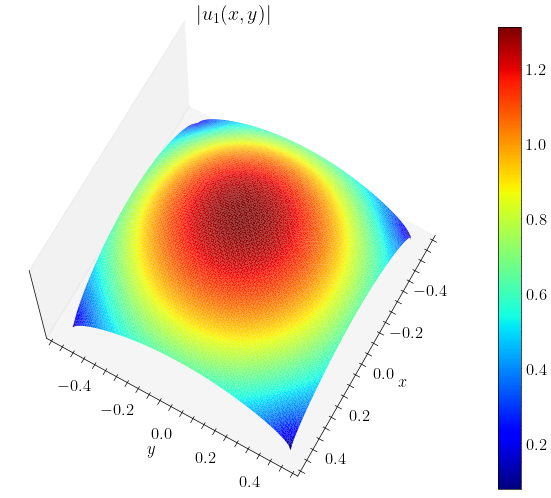}
    &
    \includegraphics[width=0.48\linewidth]{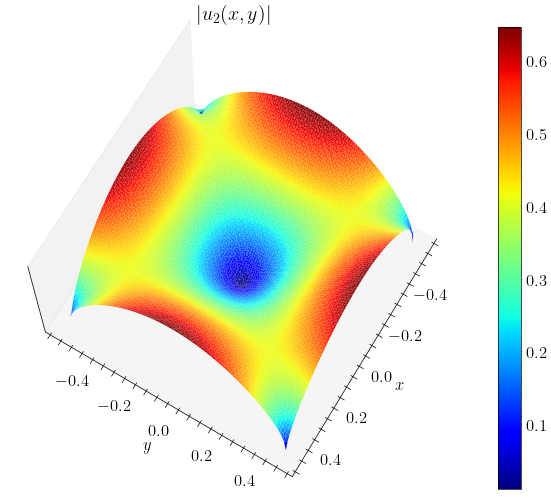}
\end{tabular}
\caption{Absolute values of the first eigenfunction components 
of the square for $m=1$.}\label{eigenfunction_abs_u_1_square}  
\end{center} 
\end{figure}

\subsection{A refined upper bound}
We complement the numerical results for rectangles
by a new analytic upper bound.
\begin{theorem}\label{Thm}
    For every $m \geq 0$, one has
    \begin{equation}\label{upper_bound}
        \lambda_1(a, b)^2 - m^2 \leq \min \left\{\left(\frac{\nu_1(m a)}{a}\right)^2 + \left(\frac{\pi}{b}\right)^2, \left(\frac{\pi}{a}\right)^2 + \left(\frac{\nu_1(m b)}{b}\right)^2\right\},
    \end{equation}
    where \(\nu_1(m)\) is the unique root of the equation 
\begin{equation}\label{implicit}
        \frac{\tan(\nu)}{\nu} = -\frac{1}{m}
\end{equation}
lying in the interval $[\frac{\pi}{2},\pi)$
(for $m=0$ we set $\nu_1(0) = \frac{\pi}{2}$ by continuity).    
\end{theorem}
\begin{figure}[h]
\begin{center}
\includegraphics[width=0.4\linewidth]{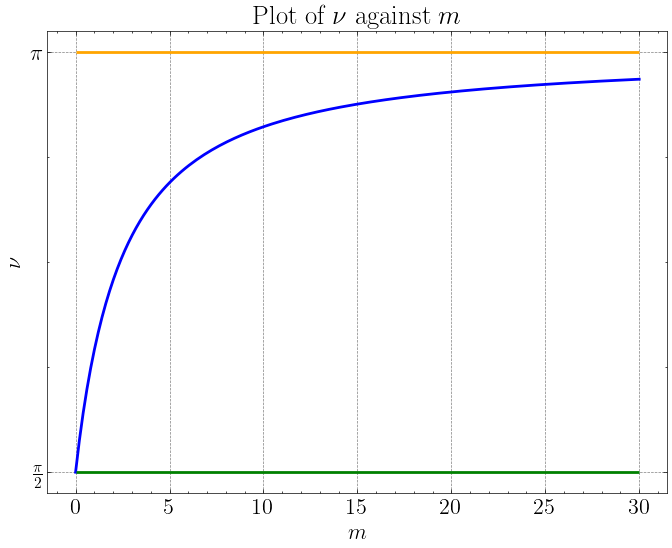}
\caption{The dependence of the solution 
of~\eqref{implicit} on~$m$.}\label{Fig.implicit}  
\end{center} 
\end{figure}

This result is a mass-dependent 
improvement upon Theorem~\ref{laplacian_bound}.
Indeed, $\nu_1(m) < \pi$ for every $m \geq 0$.
Note that $m \mapsto \nu_1(m)$ is increasing
and $\nu_1(m) \to \pi$ as as $m \to \infty$.
The dependence of $\nu_1(m)$ on~$m$ is visualised 
in Figure~\ref{Fig.implicit}.  
Figures~\ref{mass_3d}.(a), (b) and (c) compare the upper bounds of
Theorem~\ref{laplacian_bound} and~\ref{Thm}
for different rectangles with unit area. 
Both bounds become sharp in the limit $m \to \infty$, 
when $\lambda_1(a, b)^2 - m^2$ approaches the Dirichlet eigenvalue,
see~\eqref{limit}.
The advantage of~\eqref{upper_bound} 
is more evident for rectangles with high eccentricity.


%
\begin{figure}[ht]
  \centering
  \begin{tabular}{ccc}
   \includegraphics[width=0.32\textwidth]{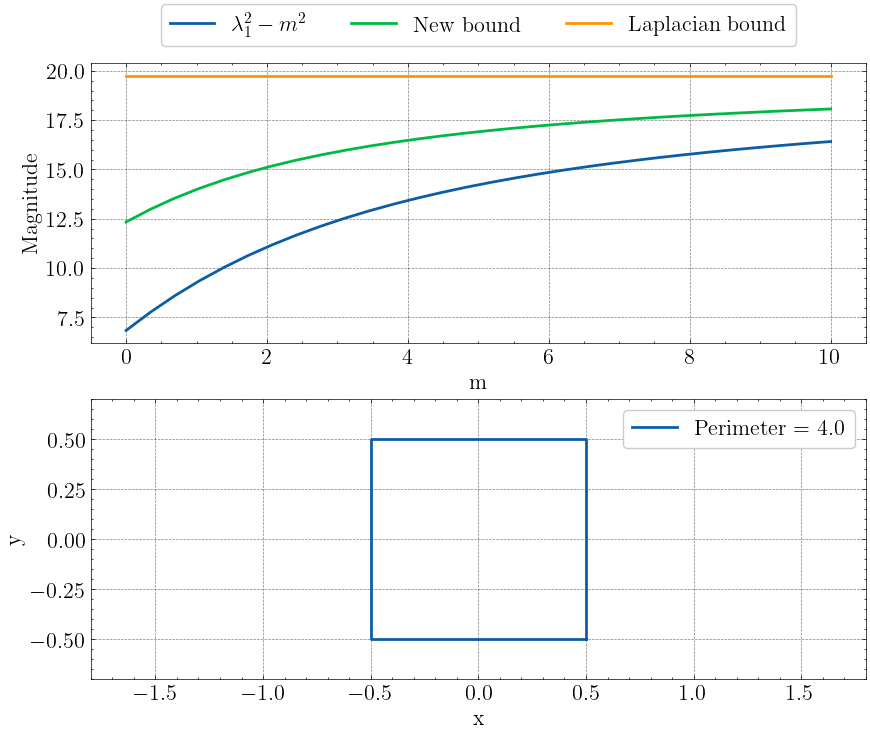}
   &\includegraphics[width=0.32\textwidth]{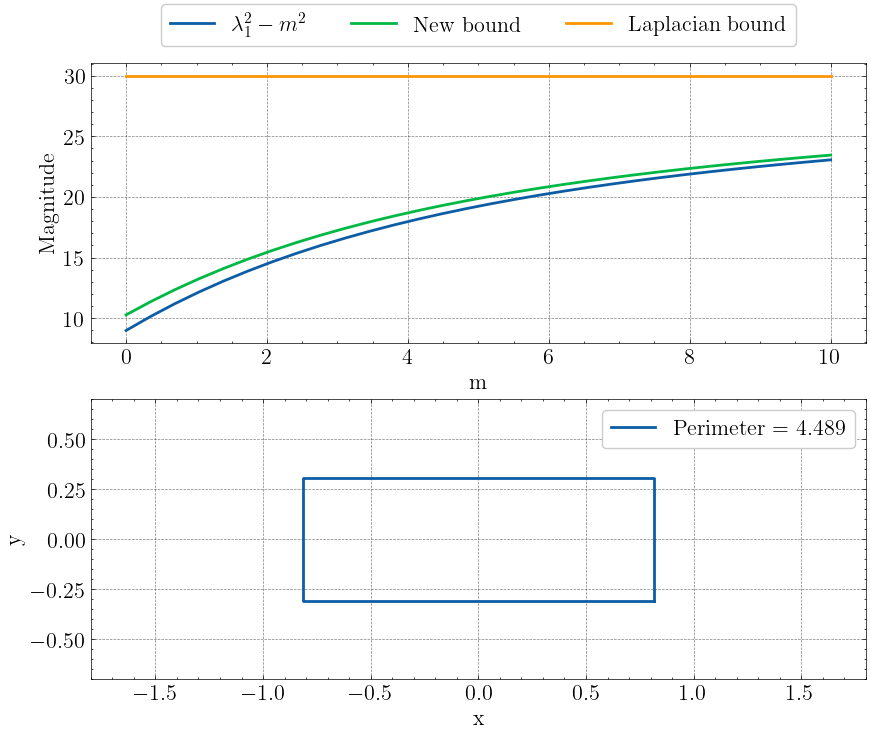}
   &\includegraphics[width=0.32\textwidth]{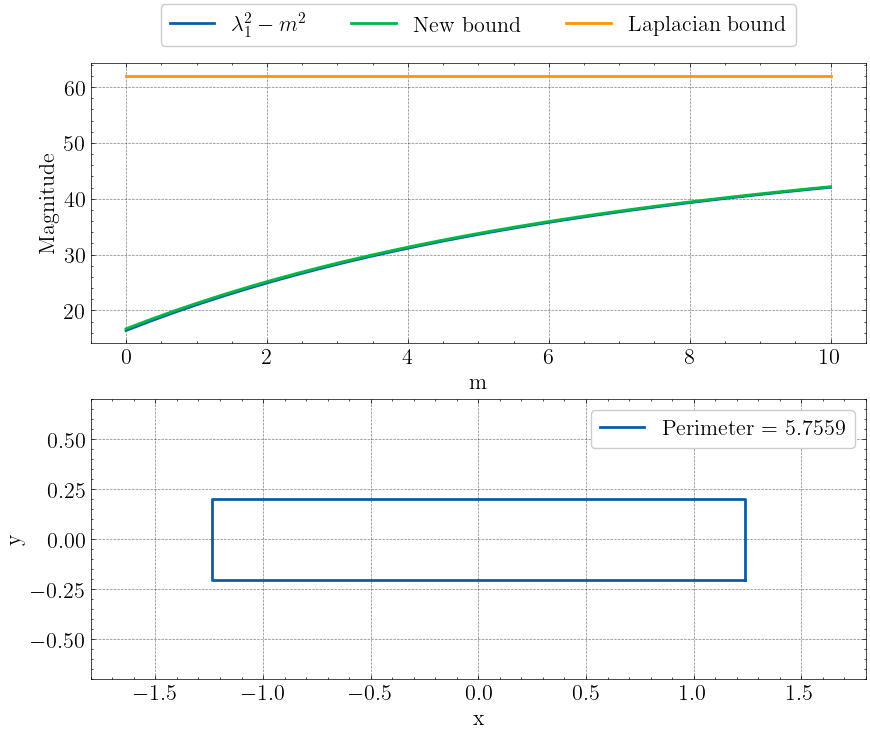}
   \\
   (a) & (b) & (c)
   \end{tabular}
\caption{Dependence of \(\lambda_1^2 - m^2\) on~$m$ and the upper bounds \eqref{laplacian_bound_eqn} and \eqref{upper_bound} for three different rectangles with varying width.}\label{mass_3d} 
\end{figure}
\begin{proof}[Proof of Theorem~\ref{Thm}]
Let $\Omega_{a,b} = (-\ademi,\ademi) \times (-\bdemi,\bdemi)$, 
a rectangle of sides~$a$ and~$b$. 
Let $H_{a,b}$ be the Dirac operator~$H$ in $L^2(\Omega_{a,b};\Com^2)$.
Its lowest positive eigenvalue~$\lambda_1(a,b)$ 
is determined by the variational formula
\begin{equation}\label{variational}  
  \lambda_1(a,b)^2 
  = \inf_u
  \frac{\|H_{a,b} u\|^2}{\|u\|^2}  
  \,,
\end{equation}
where the infimum is taken over all non-zero functions~$u$
in the domain of~$H_{a,b}$. 
More specifically, $u \in H^1(\Omega_{a,b};\Com^2)$
and the boundary conditions
\begin{equation}\label{system} 
\begin{aligned}
  u_2 &= -u_1 
  && \mbox{on} \quad \left(-\ademi,\ademi\right) \times \left\{\bdemi\right\} \,,
  \\
  u_2 &= u_1 
  && \mbox{on} \quad \left(-\ademi,\ademi\right) \times \left\{-\bdemi\right\} \,,
  \\
  u_2 &= iu_1 
  && \mbox{on} \quad \left\{\ademi\right\} \times \left(-\bdemi,\bdemi\right) \,,
  \\
  u_2 &= -iu_1 
  && \mbox{on} \quad \left\{-\ademi\right\} \times \left(-\bdemi,\bdemi\right) 
  \,.
\end{aligned}  
\end{equation}
hold in the sense of traces.
It is non-trivial but straightforward to derive the formula
\begin{equation}\label{non-trivial} 
  \|H_{a,b} u\|^2 = \|\nabla u\|^2 + m^2 \|u\|^2 + m \, \|\gamma u\|^2  
\end{equation}
valid for every $u$ in the domain of $H_{a,b}$,
where $\gamma: H^{1}(\Omega_{a,b};\Com^2) 
\to L^2(\partial\Omega_{a,b};\Com^2)$ 
denotes the Dirichlet trace. 

Simultaneously, consider the one-dimensional operator
\begin{equation}\label{operator.1D}
\begin{aligned}
  H_a	 &:= 
  \begin{bmatrix}
    m & -i \partial \\
    -i \partial & -m
  \end{bmatrix}
  \qquad \mbox{in} \qquad
  L^2\left(\left(-\ademi,\ademi\right);\Com^2\right)
  \,,
\end{aligned}  
\end{equation}
whose domain consists of functions 
$
\varphi =
  \begin{bsmallmatrix}
    \varphi_1 \\ \varphi_2 
  \end{bsmallmatrix}
  \in H^1\left(\left(-\ademi,\ademi\right);\Com^2\right)
$
satisfying the boundary conditions  
$$
  \varphi_2(\pm a) = \pm i \varphi_1(\pm a)
  \,.
$$  
Note that~$H_a$ corresponds to 
the ``longitudinal'' part of~\eqref{operator}.
At the same time, consider
the unitarily equivalent variant $\tilde{H}_b := r^*H_br$ with 
$
  r :=
  \begin{psmallmatrix}
    i & 0 \\
    0 & 1
  \end{psmallmatrix}
$,
which corresponds to the ``transversal'' part~\eqref{operator}.
The spectral problem for~$H_a$ (and~$\tilde{H}_b$)
can be solved explicitly. 
In particular, the lowest positive eigenvalue~$\lambda_1(a)$ of~$H_a$
equals 
\begin{equation}
  \lambda_1(a) = \sqrt{m^2 + \left(\frac{\nu_1(ma)}{a}\right)^2}
  \,,
\end{equation}
where $\nu_1(m)$ is the unique root of the equation~\eqref{implicit}
lying in the interval $\left[\frac{\pi}{2},\pi\right)$.
One has the variational formulation
\begin{equation}  
  \lambda_1(a)^2 
  = \inf_{\varphi} 
  \frac{\|H_{a} \varphi\|^2}{\|\varphi\|^2}  
  \,,
\end{equation}
where the infimum is taken over all non-zero functions~$\varphi$
in the domain of~$H_{a}$ and
\begin{equation}\label{non-trivial.1D} 
  \|H_{a} \varphi\|^2 
  =  \int_{-\ademi}^{\ademi} |\varphi'(x)|^2 \, \der x
  + m^2 \int_{-\ademi}^{\ademi} |\varphi(x)|^2 \, \der x
  + m \, \left(|\varphi(-\ademi)|^2 + |\varphi(\ademi)|^2\right)
  \,.
\end{equation}
The eigenfunctions of~$H_a$ can be expressed in terms
of sines and cosines~\cite{BBKO},
but we do not need the explicit formulae.

The upper bound of Theorem~\ref{laplacian_bound} 
was obtained in~\cite{briet-krej} by considering 
the trial function (associated with the first eigenfunction 
of the Dirichlet Laplacian in the rectangle)
\begin{equation*}
  u(x_1,x_2) = \cos(\mbox{$\frac{\pi}{a}$} x_1) 
  \cos(\mbox{$\frac{\pi}{a}$} x_2) 
  \begin{pmatrix}
    1 \\ 1
  \end{pmatrix}
\end{equation*}
in~\eqref{variational}.  
Note that~$u$ satisfies the Dirichlet boundary conditions
on~$\partial\Omega_{a,b}$, so~\eqref{system} automatically holds.
The improvement of Theorem~\ref{Thm} is based on the idea
that it is enough to satisfy the Dirichlet boundary conditions
on two parallel boundaries only and to use the first eigenfunction
of~$H_a$ or~$\tilde{H}_b$ in the other direction;
we are grateful to Lo\"ic Le Treust for this observation~\cite{private}.
More specifically, using the Fubini theorem,
let us re-write~\eqref{non-trivial} as follows:
\begin{equation*}
\begin{aligned}
  \lefteqn{\|H_{a,b} u\|^2 + m^2 \, \|u\|^2}
  \\
  &= \int_{-\bdemi}^{\bdemi} \left(
  \int_{-\ademi}^{\ademi} |\partial_1 u(x_1,x_2)|^2 \,  \der x_1
  + m^2 \int_{-\ademi}^{\ademi} |u(x_1,x_2)|^2 \,  \der x_1
  + m \, |u(-\ademi,x_2)|^2 + m \, |u(\ademi,x_2)|^2
  \right) \der x_2
  \\
  & \quad + \int_{-\ademi}^{\ademi} \left(
  \int_{-\bdemi}^{\bdemi} |\partial_2 u(x_1,x_2)|^2 \,  \der x_2
  + m^2 \int_{-\bdemi}^{\bdemi} |u(x_1,x_2)|^2 \,  \der x_2
  + m \, |u(x_1,-\bdemi)|^2 + m \, |u(x_1,\bdemi)|^2
  \right) \der x_1
  \,.
\end{aligned}  
\end{equation*}
Then, recalling~\eqref{non-trivial.1D},
it is clear that 
the first bound on the right-hand side of~\eqref{upper_bound}  
is obtained by using the trial function
\begin{equation} 
  u(x_1,x_2) = \varphi_1(x_1) \, \cos(\mbox{$\frac{\pi}{b}$} x_2)  
\end{equation}
in~\eqref{variational},
where~$\varphi_1$ is the eigenfunction of~$H_a$
corresponding to~$\lambda_1(a)$  
and the cosine is the Dirichlet Laplacian eigenfunction 
of the transverse interval $(-\bdemi,\bdemi)$.
The second bound on the right-hand side of~\eqref{upper_bound}  
is obtained by using the trial function
\begin{equation} 
  u(x_1,x_2) = \cos(\mbox{$\frac{\pi}{a}$} x_1) \, \tilde\varphi_1(x_2) 
\end{equation}
in~\eqref{variational},
where~$\tilde\varphi_1$ is the eigenfunction of~$\tilde{H}_b$
corresponding to~$\lambda_1(b)$.
\end{proof}

\subsection{From rectangles to quadrilaterals}
Finally, we provide a numerical evidence 
that the upper bound~\eqref{upper_bound} 
may be extended for any quadrilateral. 
We rewrite the upper bound \eqref{upper_bound} 
using the formulae for the area and perimeter of the rectangle, \(A = ab\) and \(P = 2(a+b)\), respectively. Then,
\[
    \begin{cases}
        A = ab \\
        P = 2 (a+b)
    \end{cases} 
    \iff
    \begin{cases}
        a = \frac{P \pm \sqrt{P^2 - 16 A}}{4}\\
        b = \frac{P \mp \sqrt{P^2 - 16 A}}{4}
    \end{cases}
    .
\]
For our simulations, we choose \(a = \frac{P + \sqrt{P^2 - 16 A}}{4}\) and \(b = \frac{P - \sqrt{P^2 - 16 A}}{4}\), since the choice is unimportant. 
Then, substituting in \eqref{upper_bound}, we have a formula that depends just on the area and perimeter. Our numerical results suggest that this procedure leads to an upper bound valid for any quadrilateral.

\begin{conj}\label{bound_conjecture}
  Let $Q$ be a quadrilateral with area $A$ and perimeter $P$. Then, for every $m\geq0,$ we have 
    \begin{equation}\label{upper_bound_quad}
\begin{aligned}
\lefteqn{
        \lambda_1(Q)^2 - m^2 
        }
        \\
        &&\leq  \min \left\{\left(\frac{\nu_1(m  \frac{P +  \sqrt{P^2 - 16 A}}{4})}{ \frac{P +  \sqrt{P^2 - 16 A}}{4}}\right)^2 + \left(\frac{\pi}{\frac{P -  \sqrt{P^2 - 16 A}}{4}}\right)^2, \left(\frac{\pi}{ \frac{P + \sqrt{P^2 - 16 A}}{4}}\right)^2 + \left(\frac{\nu_1(m \frac{P -  \sqrt{P^2 - 16 A}}{4})}{\frac{P -  \sqrt{P^2 - 16 A}}{4}}\right)^2\right\}
        \,.
\end{aligned}        
    \end{equation}
\end{conj}

The validity of this conjecture is supported in Figure~\ref{mass_quad},
where the first eigenvalue of randomly generated quadrilaterals 
with unit area are computed for different values of the mass 
\(m \in [0, 10]\). 
For each of the quadrilaterals~$Q$, the plot of the magnitude of 
\(\lambda_1(Q) - m^2\) against the mass is presented and accompanied by the mass-dependent upper bound~\eqref{upper_bound_quad}, as in Figures \ref{mass_3d}.(a), (b) and (c). 
For completeness, we also show the corresponding upper bound~\ref{laplacian_bound}, applying the same procedure by replacing the dependence on $a$ and $b$ by a dependence on the area and perimeter.

In Figure~\ref{mass_quad}.(b), it appears that the bound \eqref{upper_bound_quad} is almost optimal, given the proximity between the blue and green lines. In Figure~\ref{mass_quad}.(c), 
we illustrate the sensitivity of the upper bounds \eqref{laplacian_bound_eqn} and \eqref{upper_bound} to quadrilaterals with both high perimeter and non-convex shapes, resulting in a significant gap between these bounds and \(\lambda_1(Q) - m^2\). 

\begin{figure}[h]
  \centering
  \begin{tabular}{ccc}
   \includegraphics[width=0.32\textwidth]{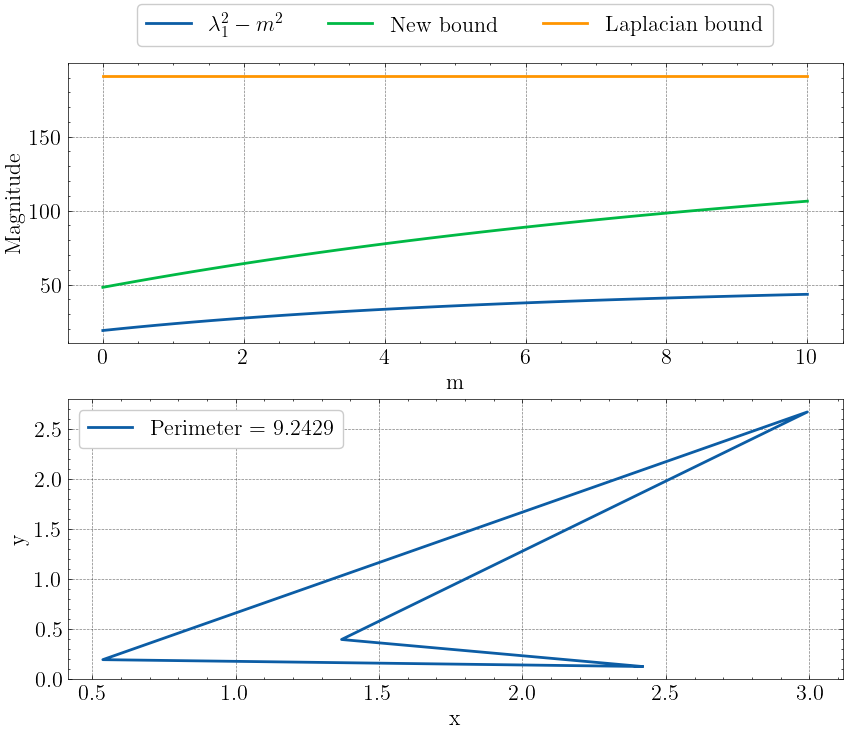}
   &\includegraphics[width=0.32\textwidth]{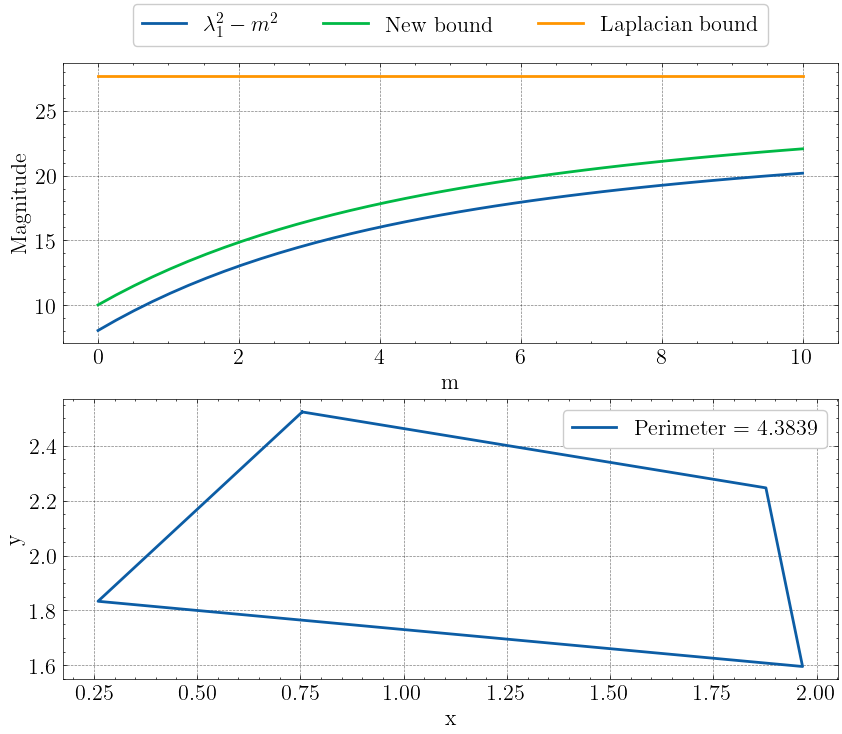}
   &\includegraphics[width=0.32\textwidth]{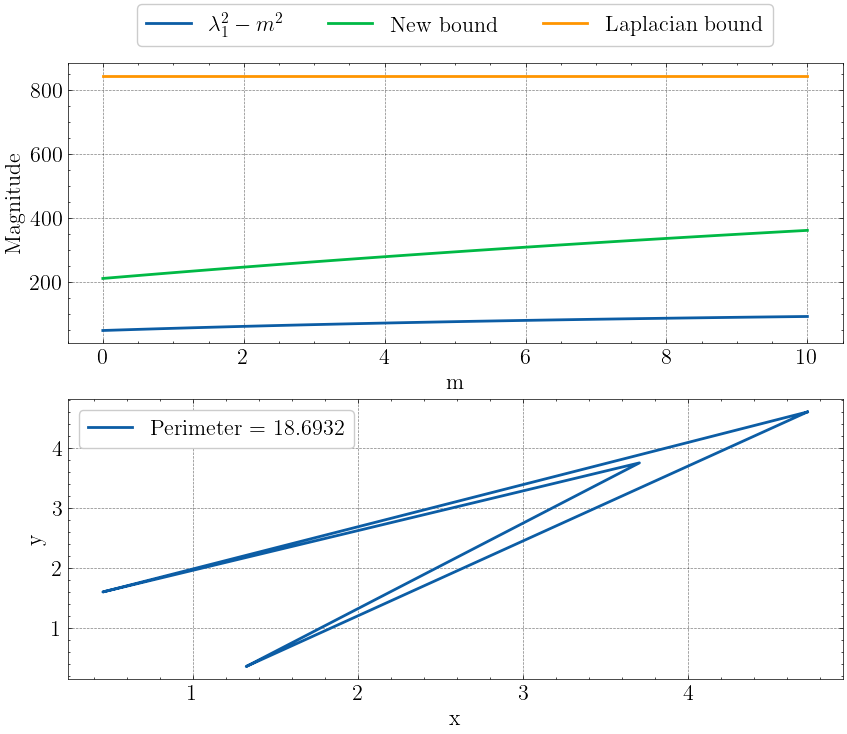}
   \\
   (a) & (b) & (c)
   \end{tabular}
\caption{Dependence of \(\lambda_1^2 - m^2\) 
and the upper bound \eqref{upper_bound_quad} on~$m$
for three quadrilaterals.}\label{mass_quad} 
\end{figure}

An interesting open problem is to show that the square
is at least a \emph{local} minimiser for the first eigenvalue
among all quadrilaterals (even just rectangles) of fixed area. 
For the Robin Laplacian with negative boundary parameter, 
an analogous question has been 
solved just recently~\cite{Larsen-Scott-Clutterbuck_2024}.
The case of triangles remains open~\cite{KLT}.

\section{Triangles}\label{Sec.tri}
%
We continue the numerical experiments 
with the Dirac operator by considering triangles.
In the non-relativistic setting, 
the approach followed in this work was previously considered in \cite{antunes2011inverse}. Instead of considering random triangles, this study can be done systematically given that, up to congruence, two parameters completely define every triangle.

\begin{defn}\label{triangle_def}
The aperture of an isosceles triangle is the angle between its equal sides. An isosceles triangle is said to be \emph{subequilateral} if the aperture is less than \(\frac{\pi}{3}\), and \emph{superequilateral} if the aperture is greater than \(\frac{\pi}{3}\). \end{defn}

\begin{figure}[h]
  \centering
  \includegraphics[width=0.45\linewidth]{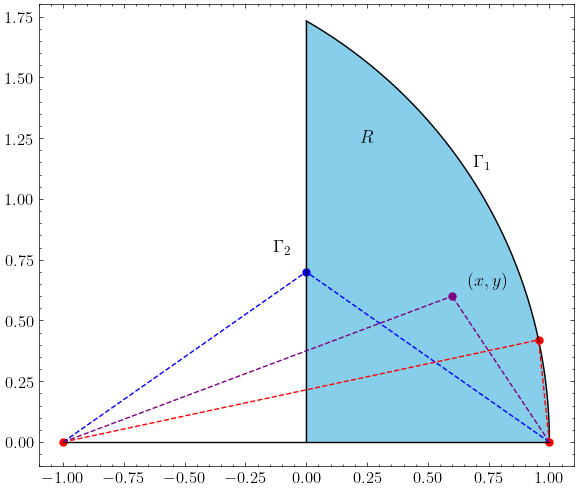}
  \caption{Configuration space of the admissible triangles. The dashed red line shows a subequilateral triangle; the dashed blue line a superequilateral triangle.}
  \label{model_triangle}
\end{figure}

Consider the \emph{admissible} region \(R\) plotted in Figure \ref{model_triangle} defined by
\[
  R = \{(x, y) \in \mathbb{R}^2: x \geq 0, y > 0, (x+1)^2 + y^2 \leq 4\}
\]
with piecewise boundary \(\partial R = \Gamma_0 \cup \Gamma_1 \cup \Gamma_2\) with
\begin{align*}
  &\Gamma_0 = \{(x, y) \in \mathbb{R}^2: 0 \leq x \leq 1, y = 0\},\\
  &\Gamma_1 = \{(x, y) \in \mathbb{R}^2: 0 \leq x < 1, y = \sqrt{4-(x+1)^2}\},\\
  &\Gamma_2 = \{(x, y) \in \mathbb{R}^2: x=0, 0 < y < \sqrt{3}\},
\end{align*}
such that given any triangle \(T\), there exists a triangle \(\hat{T}\) with vertices \((-1, 0), (1, 0)\) and the third vertex in~\(R\). 
  According to Definition \ref{triangle_def}, triangles with their third vertex in \(\Gamma_1\) are classified as subequilateral, and those with the third vertex in \(\Gamma_2\) are considered superequilateral. If the third vertex is within the interior of \(R\), then these triangles are characterised as scalene triangles. Such triangle \(\hat{T}\) is said to be \emph{admissible}. Notice that the admissible triangles do not have unit area; nevertheless, it suffices to scale each admissible triangle appropriately to a similar triangle with unit area.

As for the (relativistic) Dirac rectangles,
and in analogy with what is known in for the (non-relativistic)
Dirichlet Laplacian, we state the following conjecture.

\begin{conj}\label{Conj.tri}
  Let \(\lambda_1\) denote the first eigenvalue of~$H$ on a triangle. 
  Then
  \begin{itemize}
    \item among all triangles of fixed area,
    $\lambda_1$ is minimised by the equilateral triangle;
    \item among all triangles of fixed perimeter,
    $\lambda_1$ is minimised by the equilateral triangle. 
  \end{itemize}
\end{conj}

Figures~\ref{triangle_first_eigenvalue} 
and~\ref{triangle_3d_first_eigenvalue} show our numerical results 
for the first eigenvalue for \(m=1\) under the area constraint. 
In general, our numerical experiments support the validity 
of Conjecture~\ref{Conj.tri} 
(the isoperimetric optimality follows as a consequence
of the area constraint considered).

\begin{figure}[h]
  \centering
  \begin{minipage}{.5\textwidth}
    \centering
    \includegraphics[width=0.95\linewidth]{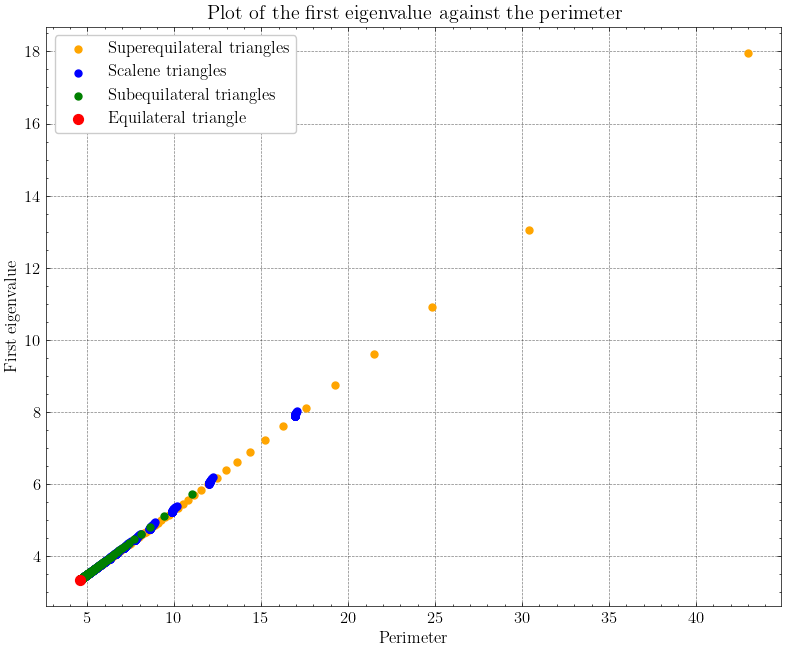}
    \captionsetup{width=0.9\linewidth}
    \captionof{figure}{Plot of the first eigenvalue against the perimeter for triangular domains.}
    \label{triangle_first_eigenvalue}
  \end{minipage}%
  \begin{minipage}{.5\textwidth}
    \centering
    \includegraphics[width=0.9\linewidth]{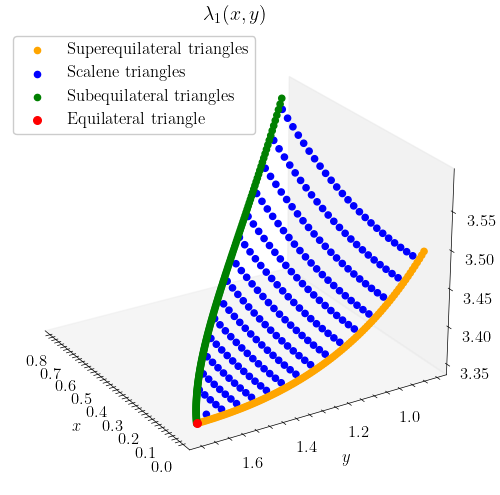}
    \captionsetup{width=0.9\linewidth}
    \captionof{figure}{Plot of the first eigenvalue \(\lambda_1(x, y)\) for each sample point \((x, y) \in R\).}
    \label{triangle_3d_first_eigenvalue}
  \end{minipage}
\end{figure} 

Figure~\ref{triangle_u} shows
the absolute value of the eigenfunction components~$u_1$ and~$u_2$
associated with the first eigenvalue in the equilateral triangle
with $m=1$.  

\begin{figure}[h]
\begin{center}
\begin{tabular}{cc}
    \includegraphics[width=0.48\linewidth]{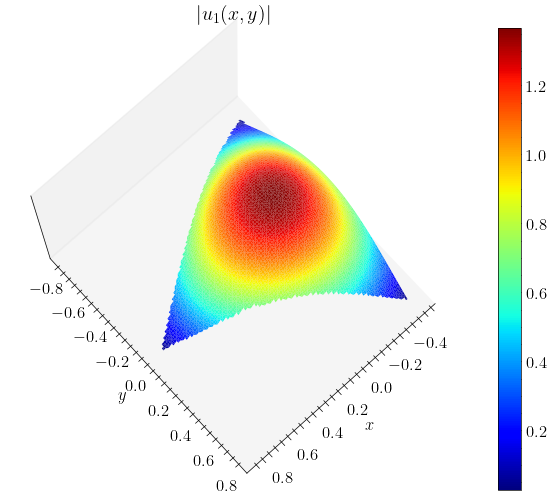}
    &
    \includegraphics[width=0.48\linewidth]{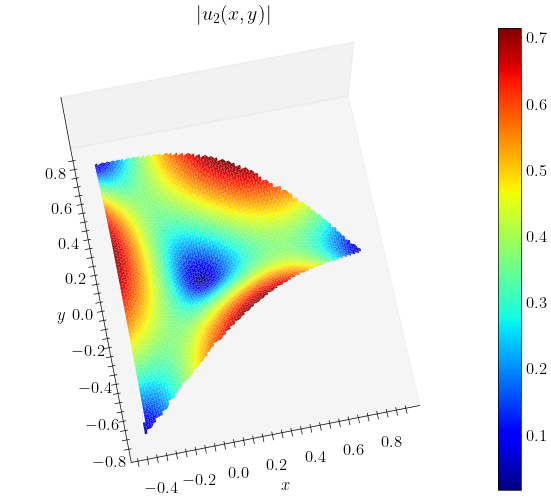}
\end{tabular}
\caption{Absolute values of the first eigenfunction components 
of the equilateral triangle for $m=1$.}\label{triangle_u}  
\end{center} 
\end{figure}
%

\section{Polygons}\label{Sec.poly}
%
Going from rectangles and triangles to arbitrary polygons,
we state the following general conjecture.

\begin{conj}\label{polygons_conjecture}
Given any $n \geq 1$,
 let $P_n$ be any $n$-sided polygon 
 and $P_n^*$ be the $n$-sided regular polygon, 
 both with fixed area (or perimeter). 
 Then \(\lambda_1 (P_n)\geq \lambda_1 (P_n^*)\).
\end{conj}

For triangles and quadrilaterals (at least rectangles),
the validity of the conjecture is supported already 
in Sections~\ref{Sec.rect} and~\ref{Sec.tri}, respectively.
Our numerical experiments support the validity 
of the conjecture also for more-sided polygons.
In particular, 
in Figures~\ref{eigen_pentagons}, \ref{eigen_hexagons}, 
\ref{eigen_heptagons} and \ref{eigen_octagons},
we show our results for convex pentagons, hexagons, heptagons
and octagons with \(m = 1\), respectively.

In Figure~\ref{vertices_magnitude}, 
we plot the first three eigenvalues of the regular \(n\)-sided polygons with fixed unit area for \(n=3,4,5,6,7,8\), with \(m=1\). 
In the dashed red lines the corresponding eigenvalues of the unit area disk are also presented. 

\begin{figure}[h]
  \centering
  \includegraphics[width=0.6\linewidth]{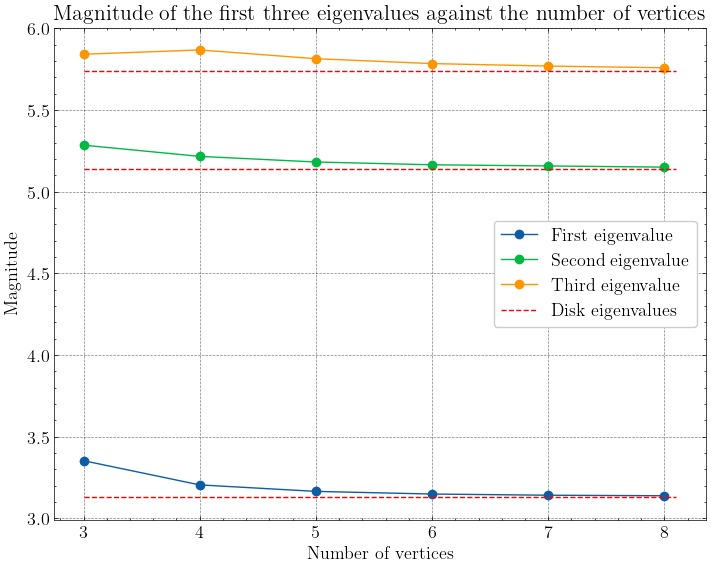}
  \caption{Comparison of the first three eigenvalues 
  of regular \(n\)-sided polygons against the disk, with fixed unit area.}
  \label{vertices_magnitude}
\end{figure}
\begin{figure}[t]
  \centering
  \begin{minipage}{.5\textwidth}
    \centering
    \includegraphics[width=0.9\linewidth]{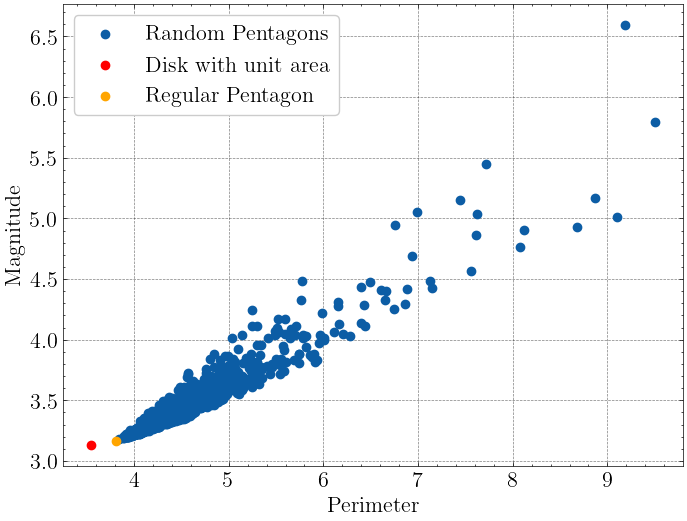}
    \captionsetup{width=0.9\linewidth}
    \captionof{figure}{Numerical simulations for the first eigenvalue of convex pentagons.}
    \label{eigen_pentagons}
  \end{minipage}%
  \begin{minipage}{.5\textwidth}
    \centering
    \includegraphics[width=0.9\linewidth]{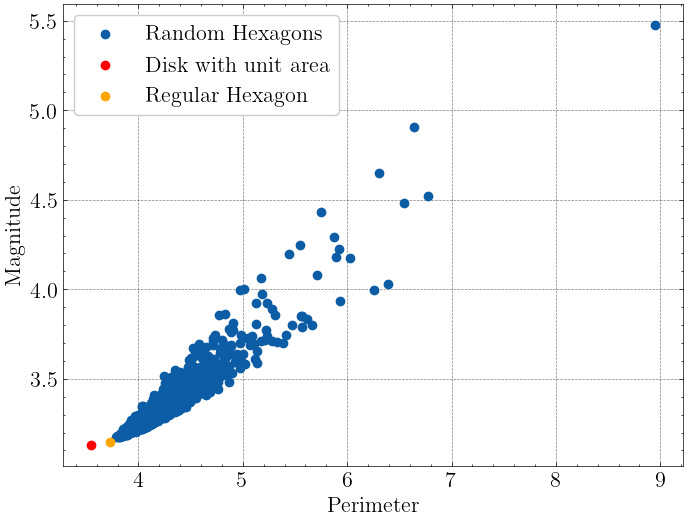}
    \captionsetup{width=0.9\linewidth}
    \captionof{figure}{Numerical simulations for the first eigenvalue of convex hexagons.}
    \label{eigen_hexagons}
  \end{minipage}
  \vfill
  \vspace*{0.5cm}
  \centering
  \begin{minipage}{.5\textwidth}
    \centering
    \includegraphics[width=0.9\linewidth]{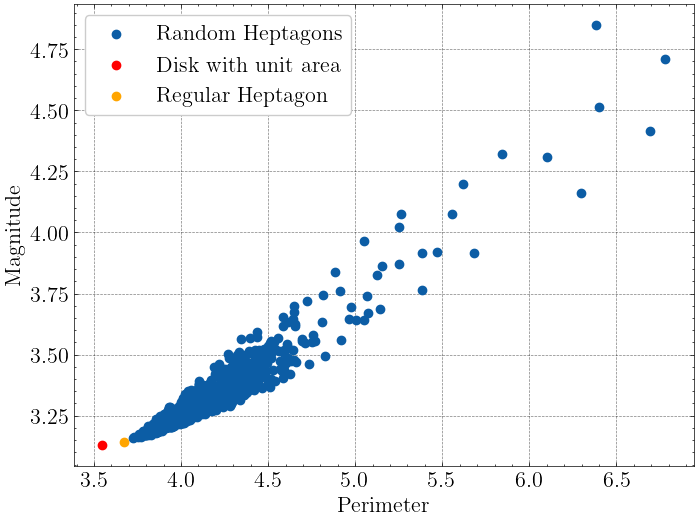}
    \captionsetup{width=0.9\linewidth}
    \captionof{figure}{Numerical simulations for the first eigenvalue of convex heptagons.}
    \label{eigen_heptagons}
  \end{minipage}%
  \begin{minipage}{.5\textwidth}
    \centering
    \includegraphics[width=0.88\linewidth]{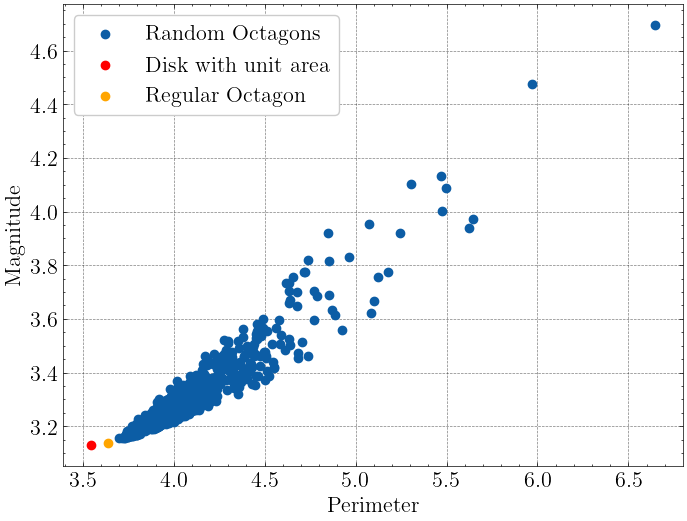}
    \captionsetup{width=0.9\linewidth}
    \captionof{figure}{Numerical simulations for the first eigenvalue of convex octagons.}
    \label{eigen_octagons}
  \end{minipage}
\end{figure}
%

\section{General domains}\label{Sec.general}
%
Finally, we consider the spectral optimisation 
of the Dirac operator for general domains
as regards the disk.
Since this problem for the first eigenvalue 
was already investigated in~\cite{AKO}
(finding, numerically, that~$\lambda_1$ is indeed minimised by the disk,
among all domains of fixed area),
we focus on higher eigenvalues.

\subsection{The second eigenvalue}\label{subsection_second_eigenvalue}
Let us start with searching for a domain with unit area 
minimising the second eigenvalue.
Let \(\mathcal{B}\) be the set of all open and bounded sets of \(\mathbb{R}^2\) with unit area. Then, we define the functional 
$$
    \mathcal{F}_2: \, \mathcal{B} \rightarrow \mathbb{R} 
    : \{\Omega \mapsto \lambda_2\left(\Omega\right) \}
$$
and consider the minimisation problem
\begin{equation}\label{min_problem_eigen_2}
  \min_{\Omega \in \mathcal{B}} \mathcal{F}_2(\Omega).
\end{equation} 

For practical purposes, we consider the map
$$
    \mathcal{P}: \, \mathbb{R}^{2M+1} \rightarrow \mathcal{B}
    : \{ (a_0,a_1,...,a_M,b_1,...,b_M) \mapsto \hat{\Omega} \},
$$
where $\hat{\Omega}$ is a domain (with unit area) whose boundary can be parameterised by
\begin{equation}
\label{new_radial}
\left\{r_N(t) \left(\cos(\theta),\sin(\theta)\right),\ \theta\in[0,2\pi[\right\},
\end{equation}
where
\begin{equation}\label{radial_param}
r_N(t)=\sqrt{\frac{2}{\pi}}\frac{a_0 + \sum_{m=1}^{M}a_m \cos(m \theta) + \sum_{m=1}^{M}b_m \sin(m \theta)}{\left(2a_0^2 + \sum_{m=1}^{M}a_m^2 + \sum_{m=1}^{M}b_m^2\right)^{\frac{1}{2}}}.
\end{equation}
It is straightforward to verify that domains with boundary parameterised in this way have always unit area, independently of the coefficients $a_m$ and $b_m$.
Then, the optimisation problem \eqref{min_problem_eigen_2} can be discretised into
\begin{equation}\label{min_problem_eigen_2_disc}
  \min_{v \in \mathbb{R}^{2M+1}} \mathcal{F}_2\left(\mathcal{P}(v)\right).
\end{equation}

Problem \eqref{min_problem_eigen_2_disc} can then be solved using direct search methods. In this work, the Nelder--Mead method is used (we refer to \cite{nelder1965simplex}). This method does not need any information on the derivative and relies solely on evaluations of the objective function in order to find the local minima. In general, one could start with a non-radial domain, not defined through the radial parameterisation \eqref{new_radial}. In that case, the coefficients \(a_0, a_1,\dots, a_M,b_1,\dots,b_M\) of such domain can be found by discretising the boundary of the domain, changing from Cartesian to polar coordinates (if needed), and using least-square methods to solve an over-determined system, yielding a radial approximation of the initial domain.

\begin{figure}[h]
  \centering
  \begin{minipage}{.5\textwidth}
    \centering
    \includegraphics[width=1\linewidth]{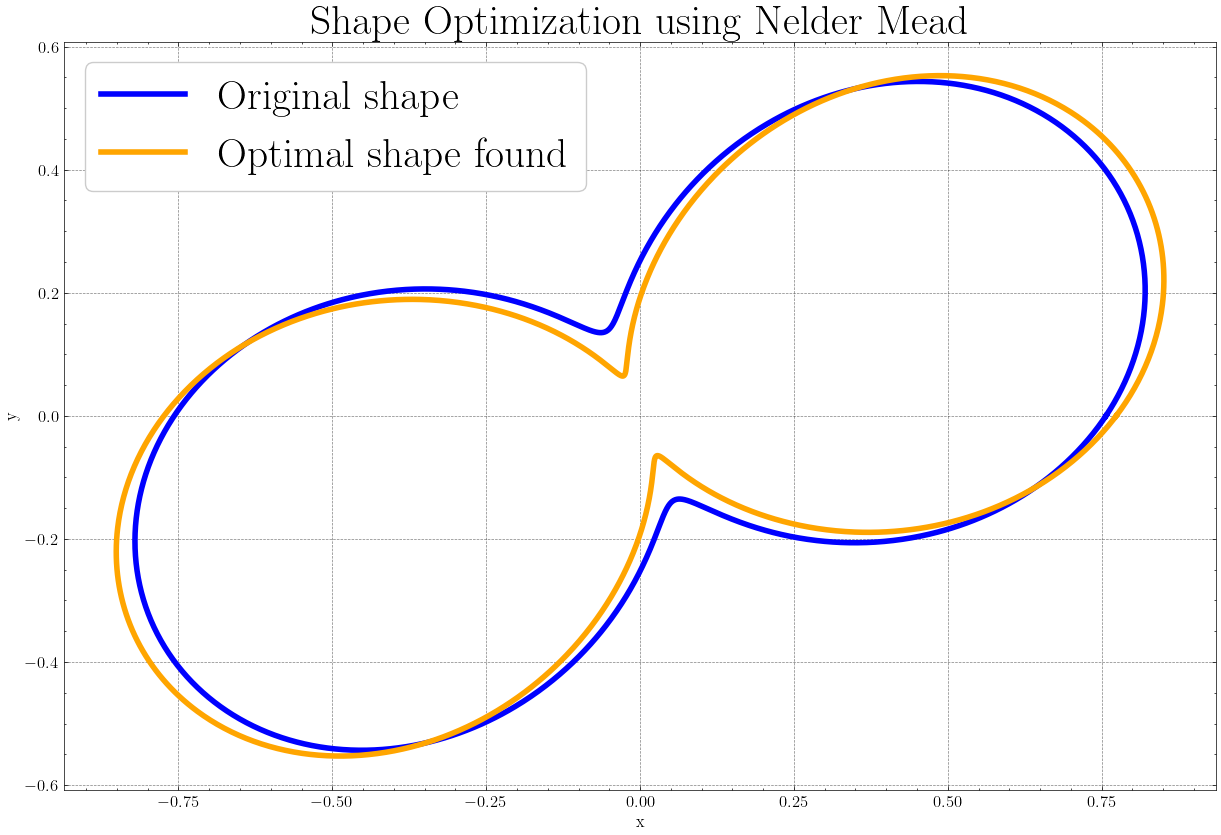}
    \captionsetup{width=0.9\linewidth}
    \captionof{figure}{Original domain used in the first iteration of the Nelder--Mead method and the optimal domain \(\Omega^\star\) found, with \(\lambda_2(\Omega^\star) \approx 4.2326443\), for \(m=1\).}
    \label{second_eigenvalue_mass_1}
  \end{minipage}%
  \begin{minipage}{.5\textwidth}
    \centering
    \includegraphics[width=1\linewidth]{Images/second_eigenvalue_mass_1.png}
    \captionsetup{width=0.9\linewidth}
    \captionof{figure}{Original domain used in the first iteration of the Nelder--Mead method and the optimal domain \(\Omega^\star\) found, with \(\lambda_2(\Omega^\star) \approx 7.1037228\), for \(m=5\).}
    \label{second_eigenvalue_mass_5}
  \end{minipage}
\end{figure}

Figures \ref{second_eigenvalue_mass_1} and \ref{second_eigenvalue_mass_5} show the optimal shape found with \(m=1\) and \(m=5\), respectively, as well as the initial shape that has the radial parameterisation
\[
  r(\theta) = 1 + \frac{\sin(2\theta)}{2} + \frac{\cos(2\theta)}{2}.
\] 
To mitigate potential dependencies on the initial shape which could influence the optimal domain, several additional simulations were carried out using different starting domains and collocation/source points for the Method of Fundamental Solutions. The domain presented not only exhibited the lowest second eigenvalue but also shared a very similar shape to the other domains identified in various simulations. Unfortunately, it is very hard to impose topological conditions on the optimisation algorithm and the simulations show that the method tries to close the gap between the cusps in Figures \ref{second_eigenvalue_mass_1} and \ref{second_eigenvalue_mass_5}, resulting in a disjoint domain. However, since we are working with only one connected domain, this forces the domain to auto-intersect, creating ties which are not addressed by the MFS. In the same vein as in the relativistic analogue of the Faber--Krahn inequality, and given the resemblance of the optimal found shape with two disjoint disks with the same area, we also conjecture that a relativistic analogue of the Krahn--Szeg\"{o} inequality also holds for the Dirac operator with infinite mass boundary conditions. For future reference, the first and second eigenvalues of two disjoint disks with unitary total area for $m=1$ are \(\lambda_1 = \lambda_2 \approx 4.16799436\). These results suggest that an analogous result to the Krahn inequality for the classical Laplacian \cite{Krahn} shall hold for Dirac eigenvalues.

\begin{conj}
The minimiser of the second eigenvalue of~$H$ among all planar bounded open sets  with unit area is the union of two disjoint disks both with area \(\frac{1}{2}\).
\end{conj}

\subsection{The third eigenvalue}
As discussed in Section~\ref{Sec.rect}, 
from Figure~\ref{eigen_rectangle_area_m_1}, is possible to verify that the square is not the global minimiser for the third eigenvalue if \(m=1\), 
among all rectangles with fixed area. 
This is in striking difference with the non-relativistic setting,
i.e.\ the Dirichlet Laplacian.
For the latter, it is conjectured that the third eigenvalue 
is minimised by the disk, among all domains of fixed area.
How is it in the relativistic setting of the Dirac operator?
 
Figures \ref{third_eigenvalue_mass_1} and \ref{third_eigenvalue_mass_1_mink} summarise the shape optimisation results after applying the Nelder--Mead algorithm, used to find the domain which minimises the third eigenvalue, as explained in Section~\ref{subsection_second_eigenvalue}. Figure \ref{third_eigenvalue_mass_1}, shows the optimal shape \(\Omega^\star\) founded for which we obtained \(\lambda_3(\Omega^\star) \approx 5.62889189\) for \(m=1\). This result 
is a strong numerical support for our conjecture that
the minimiser is not the disk \(\mathbb{D}\).
Figure \ref{third_eigenvalue_mass_1_mink} shows the plots of the first three eigenvalues of the Minkowski sum of $\mathbb{D}$ and \(\Omega^\star\), i.e, the eigenvalues of each domain \(\Omega(t)\) was computed, where \(\Omega(t)\) is given by
\[
  \Omega(t) = (1 - t)\mathbb{D} + t \Omega^\star
\]
for \(t \in [0, 1]\). 
Again, these results illustrate that the third eigenvalue  
under the area constraint is not minimised by the disk.

\begin{conj}
There exists masses $m \geq 0$ such that
the minimiser of the third eigenvalue of~$H$ among all planar bounded open sets  with fixed area is not the  disk.
\end{conj}

\begin{figure}[h]
  \centering
  \begin{minipage}{.5\textwidth}
    \centering
    \includegraphics[width=0.9\linewidth]{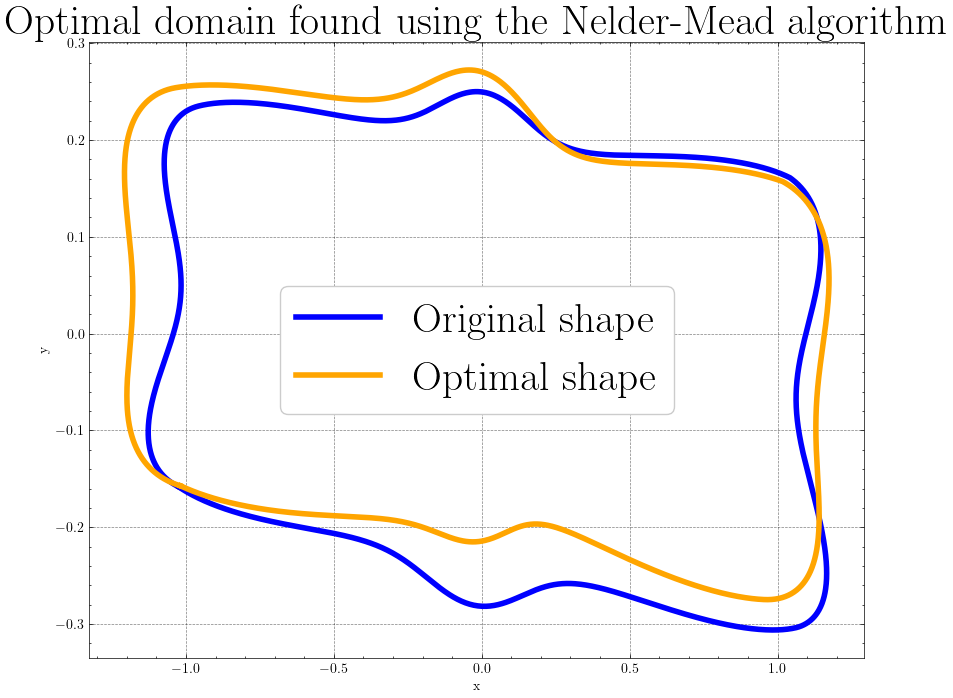}
    \captionsetup{width=0.9\linewidth}
    \captionof{figure}{Original domain used in the first iteration of the Nelder--Mead method and the optimal domain \(\Omega^\star\) found, with \(\lambda_3(\Omega^\star) \approx 5.62889189\).}
    \label{third_eigenvalue_mass_1}
  \end{minipage}%
  \begin{minipage}{.5\textwidth}
    \centering
    \includegraphics[width=1\linewidth]{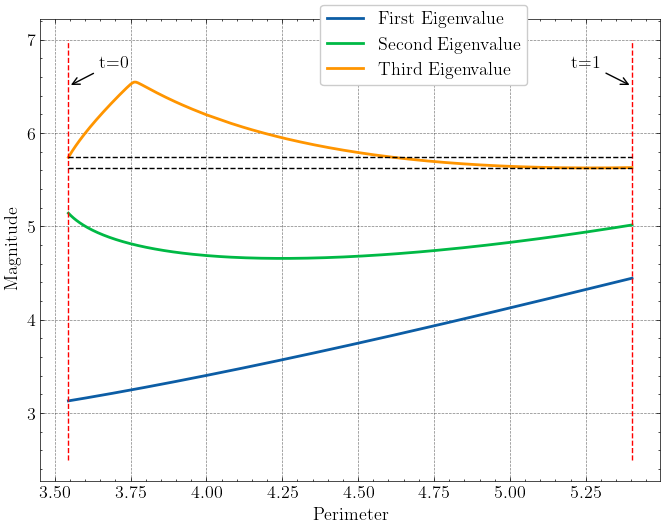}
    \captionsetup{width=0.9\linewidth}
    \captionof{figure}{The first three eigenvalues of each \(\Omega(t), \, t \in [0, 1]\).}
    \label{third_eigenvalue_mass_1_mink}
  \end{minipage}
\end{figure}

\subsection{The ratio of the second to the first eigenvalue}
Motivated by the Ashbaugh--Benguria result~\cite{ash-ben} 
in the non-relativistic setting, 
we also investigate the relativistic ratio \(\frac{\lambda_2}{\lambda_1}\). 
Figures~\ref{ab_plot} and~\ref{ab_m_5plot}
show the optimal domain which maximises the ratio \(\frac{\lambda_2}{\lambda_1}\), for which we obtained the optimal values of \(1.6428571\) and \(1.242146\) for \(m=1\) and \(m=5\), respectively. These results support the conjecture that the maximiser shall be the disk for which we have \(\frac{\lambda_2(\mathbb{D})}{\lambda_1(\mathbb{D})} \approx 1.642860118\) and \(\frac{\lambda_2(\mathbb{D})}{\lambda_1(\mathbb{D})} \approx 1.242146\) for \(m=1\) and \(m=5\), respectively.

\begin{figure}[h]
  \centering
  \begin{minipage}{.5\textwidth}
    \centering
    \includegraphics[width=0.75\linewidth]{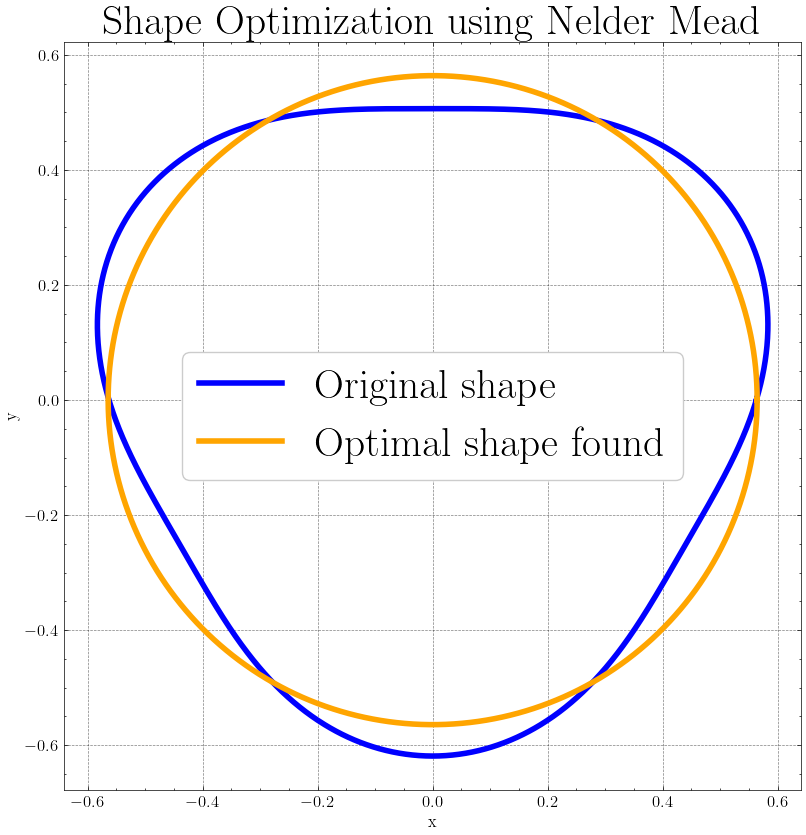}
    \captionsetup{width=0.8\linewidth}
    \captionof{figure}{Original domain used in the first iteration of the Nelder--Mead method and the optimal domain \(\Omega^\star\) found, with \(\frac{\lambda_2(\Omega^\star)}{\lambda_1(\Omega^\star)} \approx 1.6428571\) for \(m=1\).}
    \label{ab_plot}
  \end{minipage}%
  \begin{minipage}{.5\textwidth}
    \centering
    \includegraphics[width=0.775\linewidth]{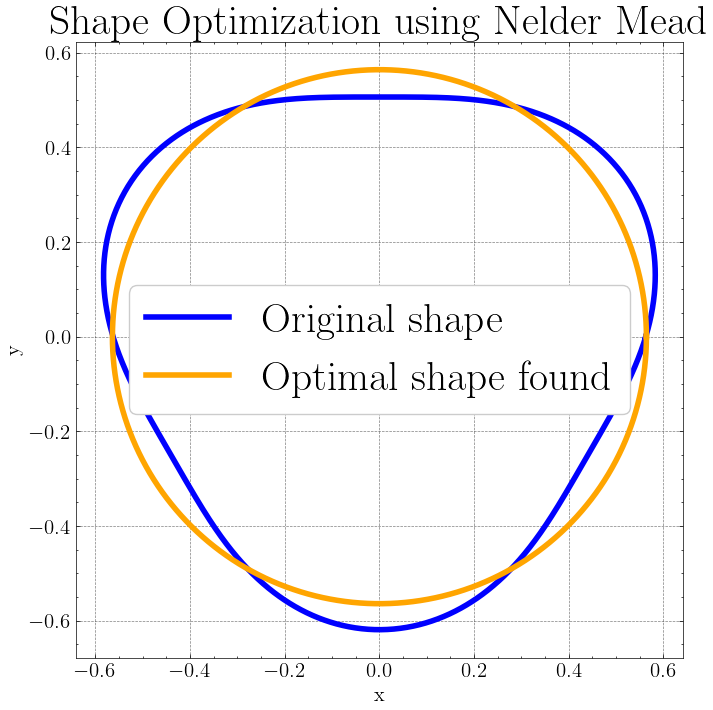}
    \captionsetup{width=0.8\linewidth}
    \captionof{figure}{Original domain used in the first iteration of the Nelder--Mead method and the optimal domain \(\Omega^\star\) found, with \(\frac{\lambda_2(\Omega^\star)}{\lambda_1(\Omega^\star)} \approx 1.242146\) for \(m=5\).}
    \label{ab_m_5plot}
  \end{minipage}
\end{figure}

\begin{conj}\label{ab_conj}
  The disk is the solution of the maximisation problem
  \[
    \max\left\{\frac{\lambda_2(\Omega)}{\lambda_1(\Omega)}:
    \ \Omega \subset \mathbb{R}^2\ \text{open and bounded}\right\}.
  \]
\end{conj}

\subsection{The ratio of the third to the first eigenvalue}
Finally, we study the maximisation of the ratio \(\frac{\lambda_3}{\lambda_1}\). Figures \ref{l-3_l-1_plot} and \ref{l-3_l-1_m_5_plot} show the optimal domain which maximises this ratio for \(m=1\) and \(m=5\), respectively. The ratio for the disk~$\mathbb{D}$ with unit area is approximately \(\frac{\lambda_3(\mathbb{D})}{\lambda_1(\mathbb{D})} \approx 1.834112925\) and \(\frac{\lambda_3(\mathbb{D})}{\lambda_1(\mathbb{D})} \approx 1.2671758\) for \(m=1\) and \(m=5\), respectively. 
The results for the Dirac operator with infinite mass boundary conditions are similar to the ones for the Dirichlet Laplacian, since it is also known that this ratio is not maximised for the disk and the optimal domain is a peanut-like shape (see, e.g., \cite{Ant,Levitin,osting}).

\begin{figure}[h]
  \centering
  \begin{minipage}{.5\textwidth}
    \centering
    \includegraphics[width=0.75\linewidth]{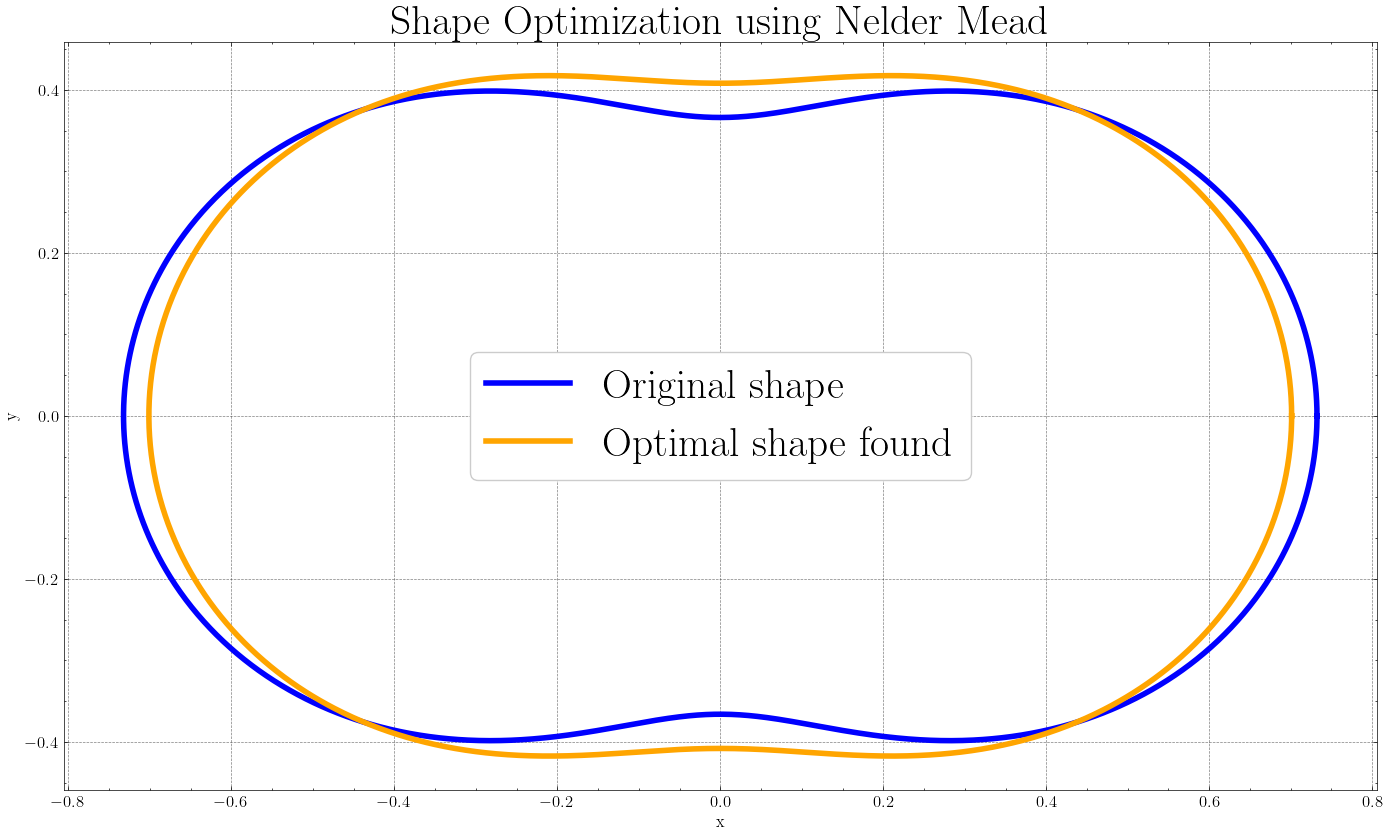}
    \captionsetup{width=0.9\linewidth}
    \captionof{figure}{Original domain used in the first iteration of the Nelder--Mead method and the optimal domain \(\Omega^\star\) found, with \(\frac{\lambda_3(\Omega^\star)}{\lambda_1(\Omega^\star)} \approx 2.0056993\), for \(m=1\).}
    \label{l-3_l-1_plot}
  \end{minipage}%
  \begin{minipage}{.5\textwidth}
    \centering
    \includegraphics[width=0.75\linewidth]{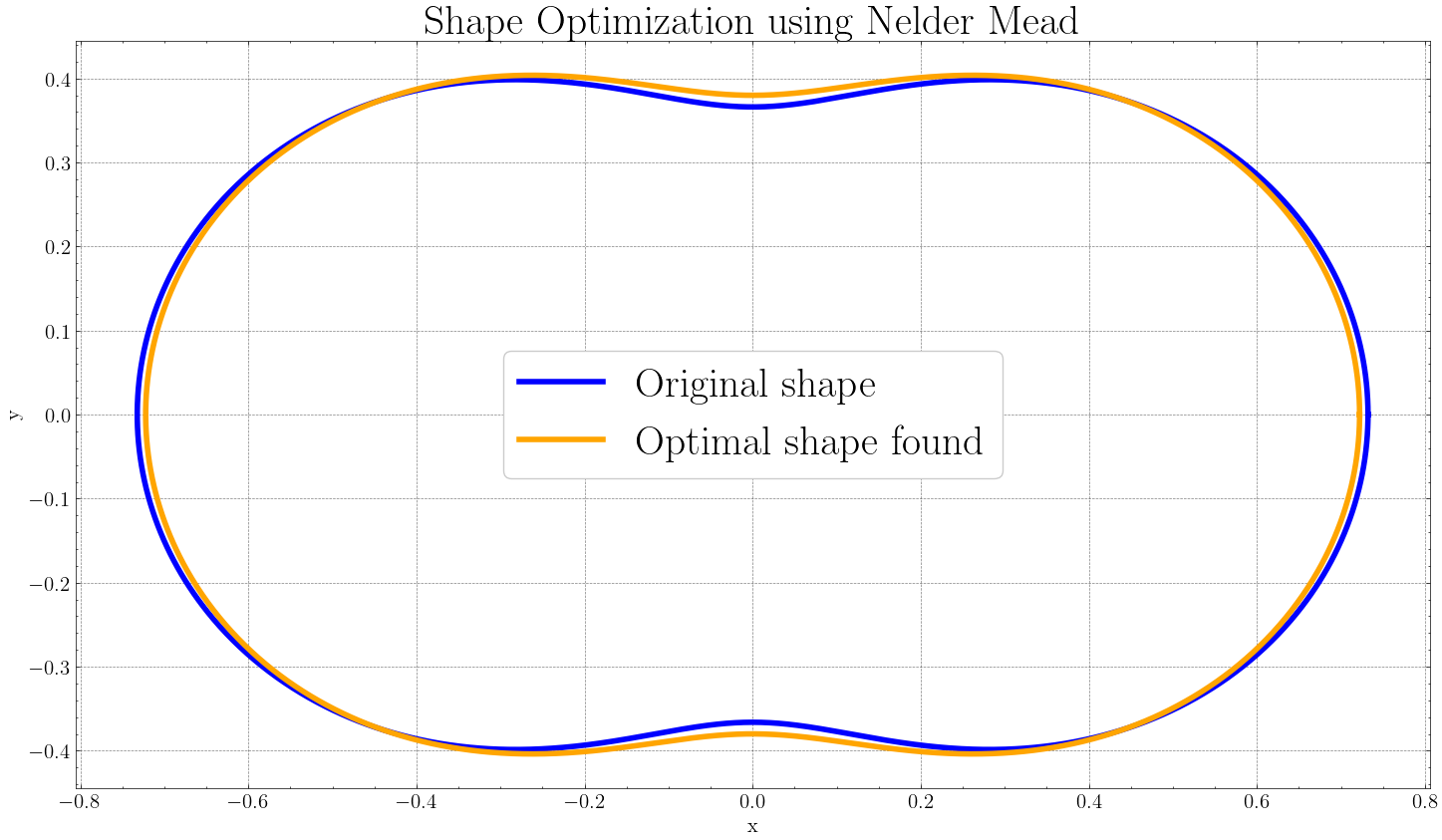}
    \captionsetup{width=0.9\linewidth}
    \captionof{figure}{Original domain used in the first iteration of the Nelder--Mead method and the optimal domain \(\Omega^\star\) found, with \(\frac{\lambda_3(\Omega^\star)}{\lambda_1(\Omega^\star)} \approx 1.386284\), for \(m=5\).}
    \label{l-3_l-1_m_5_plot}
  \end{minipage}
\end{figure}

\subsection*{Acknowledgment}
We are grateful to Lo\"ic Le Treust for the idea 
leading to Theorem~\ref{Thm}. P.A. was partially supported by FCT, Portugal, through the scientific project UIDB/00208/2020.
D.K.\ was supported
by the EXPRO grant No.~20-17749X
of the Czech Science Foundation.

\newpage


\end{document}